\documentclass[12pt]{article}   
 \usepackage{amsmath,amssymb,latexsym,footnote}
 \usepackage{graphicx,pstricks}%,showkeys}

\makeatletter  
\def\@@insvline#1#2{{\setbox0\hbox{\m@th$#1\mathrm I$}  
  \rlap{\m@th$#1 \mkern 5mu  
  \vrule height.95\ht0 depth-.005\ht0 width.09\ht0 $}  
  {\mathrm #2} }}

\def\Q{\mathpalette\@@insvline{Q}}

\makeatother  
  \newtheorem{defi}{Definition}
  \newcommand{\bd}{\begin{defi}} 
  \newcommand{\ed}{\end{defi}}

  \newtheorem{rem}[defi]{Remark}  

  \newtheorem{lemm}[defi]{Lemma}  
  \newcommand{\bl}{\begin{lemm}}
  \newcommand{\el}{\end{lemm}} 

  \newtheorem{theo}[defi]{Theorem}
  \newcommand{\bt}{\begin{theo}}
  \newcommand{\et}{\end{theo}}

  \newtheorem{cor}[defi]{Corollary}
  \newcommand{\bc}{\begin{cor}}
  \newcommand{\ec}{\end{cor}}

  \newtheorem{pro}[defi]{Proposition}
  \newcommand{\bp}{\begin{pro}}
  \newcommand{\ep}{\end{pro}}
  \makeatletter
  \def\proof{\@ifnextchar[\opargproof{\opargproof[\bf Proof \hfil\\ ]}}
  \def\opargproof[#1]{\par\noindent {\bf #1 }}
  \def\endproof{{\unskip\nobreak\hfil\penalty50\hskip8mm\hbox{}
  \nobreak\hfil
  \(\Box\)\parfillskip=0mm \par\vspace{3mm}}}
  \makeatother
\DeclareFontFamily{OT1}{nice}{}
\DeclareFontShape{OT1}{nice}{m}{n}{<5> <6> <7> <8> <9> <10>
<12><10.95><14.4><17.28><20.74><24.88>callig15}{}
\DeclareFontFamily{U}{nice}{}
\DeclareFontShape{U}{nice}{m}{n}{<5> <6> <7> <8> <9> <10>
<12><10.95><14.4><17.28><20.74><24.88>callig15}{}
\DeclareSymbolFont{calligra}{U}{nice}{m}{n}
\DeclareSymbolFontAlphabet{\nice}{calligra}
\DeclareFontFamily{OT1}{cmdh}{}
\DeclareFontShape{OT1}{cmdh}{m}{n}{<10>cmdunh10}{}
\input umsa.fd
\input umsb.fd
\input ueuex.fd
\input ueuf.fd
\input ueur.fd
\input ueus.fd

\def\epsilon{\varepsilon}  
\def\phi{\varphi}

\def\div{\operatorname{div}}

\begin{document} 

\title{Analysis of a turbulence model related to that of $k-\epsilon$ for stationary and compressible flows}
\bigskip

\date{\today}

\author{P. Dreyfuss\thanks{Laboratoire J.A. Dieudonn\'e - Universit\'e de Nice Sophia Antipolis 
- Parc Valrose 06108 Nice (email: dreyfuss@unice.fr).}} 

\maketitle

\abstract{We shall study a turbulence model arising in compressible fluid mechanics. The model called $\theta - \phi$ we study is closely related to the $k-\epsilon$ model. We shall establish existence, positivity and regularity results in a very general framework.
} \noindent

\begin{keywords}
turbulence modelling, $k-\epsilon$ model, compressible flows, Stampacchia estimates. \\ 
2010 MSC: 35J70, 35J75, 76F50, 76F60
\end{keywords}

%\footnotetext{1991 Mathematics Subject 
%Classification: 35J70,35B65}
%\bigskip\bigskip

\section{Introduction}

We shall first recall some basic ideas concerning the modelisation of the turbulent fluids, the reader can consult \cite{moha,wilcox} for 
a more detailled introduction. \\ 

Let ${\bf u},p,\rho,T$ be the velocity, pressure, density and temperature of a newtonien compressible fluid. Let also 
$\tilde{\Omega}\subset \mathbb R^3$ a domain which is asummed to be bounded. Then the motion of 
the flow in $\tilde{\Omega}$ at a time $t\in \mathbb R^+$ can be described 
by the compressible Navier Stokes equations (see system (C) page 8 in \cite{lions}).
It is well known that direct simulation based on such a model is harder or even impossible at high reynolds numbers. The reason is 
that too many points of discretization are necessary and so only very simple configurations can be handled. \\ 

Thus engineers and physicists have proposed new sets of equations to describe the average of a turbulent flow. The most famous one is the 
$k-\epsilon$ model, introduced by Kolmogorov \cite{kolmo}. We shall briefly present its basic principles in the following.  
Let $v$ denote a generic physical quantity subject to turbulent (i.e. unpredictable at the macroscopic scale), we introduce 
its mean part (or its esperance) $\langle v\rangle$ by setting: $$\langle v(x,t)\rangle=\int_{\mathcal{P}} v(\omega,x,t) d\mathcal{P}(\omega),$$
where the integral is taken in a probalistic context which we shall not detail any more here. Note however that  
the operation $\langle .\rangle$ is more generally called a filter. The probalistic meaning is one but not the only possible filter 
(see for instance \cite{moha} chap.3). We shall then consider the decomposition: $v = \langle v\rangle + v'$, where $v'$ is refered to the non computable or the non relevant part and 
$\langle v\rangle$ is called the mean part (i.e. the macroscopic part). \\
 
The principle of the $k-\epsilon$ model is to describe the mean flow in terms of the mean quantities 
$\langle{\bf u}\rangle,\langle p\rangle,\langle\rho\rangle,\langle T\rangle$ together 
with two scalar functions $k$ and $\varepsilon$, which contains relevant informations about the small scales processes (or the turbulent processes). 
The variable $k$ (SI: [$\frac{m^2}{s^2}$]) is called the turbulent kinetic energy and $\epsilon$ [$\frac{m^2}{s^3}$] is the rate of dissipation of the kinetic energy. 
They are defined by: 
\begin{equation} 
k=\frac{1}{2}|{\bf u}'|^2 \qquad \varepsilon=\frac{\nu}{2}\langle |\nabla {\bf u}'+ (\nabla {\bf u}')^T|^2\rangle, \label{def-phitheta}
\end{equation}
where $\nu$ is the molecular viscosity of the fluid.
The model is then constructed by averaging (i.e. by appling the operator $\langle .\rangle$ on) the Navier-Stokes equations. Under appriopriate 
assumptions (i.e. the Reynolds hypothesis in the incompressible case, and the Favre average in the compressible case) we obtain 
a closed system of equations for the variables $\langle {\bf u}\rangle,\langle p\rangle,\langle\rho\rangle,\langle T\rangle,k$ and $\varepsilon$ (see \cite{moha} pages 61-62 for the 
incompressible case, and pages 116-117 in the compressible situation). \\

Here we shall focus on the equations for $k$ and $\epsilon$ and we consider that the others quantities $\langle {\bf u}\rangle ,\langle p\rangle,\langle\rho\rangle,\langle T\rangle$
 are known. Moreover, in order to simplify the readability we do not use the notation $\langle .\rangle$, i.e. in the 
sequel we will write ${\bf u}$ instead of $\langle {\bf u}\rangle$ and $\rho$ instead $\langle \rho \rangle$ to represent the mean velocity and density of the fluid. 
The equations for $k$ and $\epsilon$ are of convection-diffusion-reaction type:    
\begin{align}
&\partial_t k + {\bf u} \cdot \nabla k - \frac{c_{\nu}}{\rho} \text{div }((\nu +\rho\frac{k^2}{\varepsilon})\nabla k)=
c_\nu \frac{k^2}{\epsilon}F-\frac{2}{3}kD-\epsilon, \label{k1}\\ 
&\partial_t \epsilon + {\bf u} \cdot \nabla \epsilon - \frac{c_{\epsilon}}{\rho} \text{div }((\nu+\rho\frac{k^2}{\varepsilon})\nabla \varepsilon)=
c_1kF-\frac{2c_1}{3c_\nu}\epsilon D-c_2\frac{\epsilon^2}{k}, \label{eps1}
\end{align}
where $D(x,t):=\div {\bf u}(x,t), \ F(x,t):=\frac{1}{2}|\nabla {\bf u}+ (\nabla {\bf u})^T|^2-\frac{2}{3}D(x,t)^2\geq 0$ (see subsection \ref{appendixIII} in the Appendix)  and $c_\nu, c_\epsilon, c_1, c_2$ are generally taken as positive constants (see (\ref{val1}) in the Appendix).  
 
Note that equations (\ref{k1})-(\ref{eps1}) are only valid sufficiently far from the walls. In fact, in the 
vicinity of the walls of the domain $\tilde{\Omega}$, there is a thin domain $\Sigma$, called logarithmic layer in which the modulus of the velocity goes from 0 to $\mathcal{O}(1)$. In this layer we can use some wall law or a one equation model (see \cite{moha} chap.1 and \cite{moha2}) instead of 
(\ref{k1})-(\ref{eps1}). Note however that the equations (\ref{k1})-(\ref{eps1}) can be considered even in the logarithmic layer if we 
allow the coefficients $c_{\nu},c_{\epsilon},c_1$ and $c_2$ to depend appriopriately on some local Reynolds numbers (see \cite{moha} pages 59-60 and page 115). In this last situation the system is called Low-Reynolds number $k$-$\epsilon$ model. \\ 

In the following we focus on the study in the domain $\Omega:=\tilde{\Omega}\setminus\Sigma$ and we assume that its boundary $\partial \Omega$ is 
Lipschitz\footnote{see \cite{evans} p. 127}. We denote by ${\bf n}(x)$ the outward normal defined for all points $x\in\partial \Omega$. 
The boundary conditions for $k$ and $\epsilon$ on $\partial \Omega$ are well understood. We have: 
\begin{equation} 
k = k_0 \quad \text{ and } \quad \epsilon=\epsilon_0 \quad \text{ on } \partial \Omega, \label{ke-bord}
\end{equation}
where $k_0$ and $\epsilon_0$ are strictly positive functions which can be calculated by using a wall law (see \cite{moha} p.59) or 
a one equation model (see \cite{moha2}). In the following we assume that $k_0$ and $\epsilon_0$ are given. Moreover we can assume (see again 
\cite{moha} p.59) that:
\begin{equation} 
{\bf u}\cdot{\bf n} = 0 \text{ on } \partial \Omega. \label{u-bord}
\end{equation} 

We shall concentrate in this paper on a reduced system called $\theta - \phi$ model. The new variables 
$\theta$ [s] and $\phi$ [$m^{-2}$] are obtained from $k$ and $\epsilon$ by the formulas: 
\begin{equation} 
\theta=\frac{k}{\epsilon} \qquad \phi=\frac{\epsilon^2}{k^3}. \label{ke-def}
\end{equation}
These variables have a physical meaning (see \cite{wilcox}): $\theta$ represents a characteristic time of turbulence and 
$L=\phi^{-1/2}$ is a characteristic turbulent length scale. 
By using this change of variable in the equations (\ref{k1})-(\ref{eps1}) and after considering some modelisation arguments for the 
diffusion processes (see the Annexe) we obtain:
$$
(P)  \quad\left\{
	\begin{array}{l} 
	\partial_t \theta + {\bf u} \cdot \nabla \theta -\frac{1}{\rho}\div\left( (\nu+\frac{c_\theta\rho}{\theta \phi}) \nabla \theta\right) = -c_3 \theta^2 F + c_4 \theta D + c_5\\ 
     \partial_t \phi + {\bf u} \cdot \nabla \phi -\frac{1}{\rho}\div\left( (\nu+\frac{c_\phi \rho}{\theta \phi}) \nabla \phi\right) = 
     -\phi\big( c_6\theta F - c_7 D + c_8 \theta^{-1}\big)
	\end{array}
	\right. 
$$ 
where the coefficients $c_\theta,c_\phi$ and $c_i$ are all positive.

Problem $(P)$ is known as the $\theta-\varphi$ model. It differs from the $k-\epsilon$ one only by the diffusive parts and it is attractive 
by some stronger mathematical properties. Another model closely related to these systems, and having some popularity, is the $k-\omega$ one (with $\omega=\theta^{-1}$, see for instance \cite{wilcox}). \\ 
In the papers \cite{Lew,marmol} the authors have established the existence of a weak solution for (P) and a property of positiveness. This last feature takes the model useful in practice: it can be used directly or also as an intermediary stabilization procedure to the $k-\epsilon$ one (see \cite{moha-num}). Another important property attempted for a turbulence model is its 
capability to predict the possible steady states. In the previous works  only the evolutive version of  (P) was studied (except in \cite{marmol2} where however only the incompressible situation, with pertubated viscosities was considered), and 
the results obtained cannot predict the existence or non-existence of steady states. \\ 

Hence in this paper we shall study the stationary version of (P) on a bounded domain $\Omega \subset \mathbb R^N$, N=2 or 3, on which we impose the boundary conditions 
$\theta = a, \phi = b$ on $\partial \Omega$.   
Remark that by using (\ref{ke-bord}) together with (\ref{ke-def}) we obtain 
\begin{equation} 
a=\frac{k_0}{\epsilon_0} \quad b=\frac{\epsilon_0^2}{k_0^3}. \label{def-ab}
\end{equation}
Hence we can assume that $a$ and $b$ are strictly positive given functions. \\ 

We shall establish existence, positivity and regularity results in a very general framework.
In \cite{dreyfuss} we established existence and regularity results for a turbulent circulation model involving 
${\bf u}$ and $k$ as unknowns.
The reader interested for recent references concerning the numerical simulation of turbulent fluid can consult \cite{sagaut}. 

\section{Main results}

\subsection{Assumptions and notations}
 
Let (Q) denote the stationary system associated to (P). For simplicity we introduce the new parameters $C_{\text{ind}}:= \rho c_{\text{ind}}$ where the subscript 'ind' takes the integer values 3,4,5,6,7,8 or the letters $\theta$ and $\phi$. 
Then our main model (Q) has the following form: \noindent \\ 
%\begin{breakbox}
 $$
(\text{Q})  \quad\left\{
	\begin{array}{l} 
	\rho{\bf u} \cdot \nabla \theta -\div\left( (\nu+\frac{C_{\theta}}{\theta \phi}) \nabla \theta\right) = - C_3  F \theta^2 + C_4 \theta D + C_5 \quad \text{in } \Omega \\ 
     \rho {\bf u} \cdot \nabla\phi -\div\left( (\nu+\frac{C_{\phi}}{\theta\phi}) \nabla \phi\right) = 
      -\varphi(C_6 \theta F - C_7 D + C_8 \theta^{-1}) \quad \text{in }\Omega \\ 
	 \theta = a, \phi = b \quad \text{on } \partial \Omega 
	\end{array}
	\right. 
$$
%\end{breakbox} \noindent \\  
For physical reasons we are only interested in positive solutions $(\theta,\phi)$ for (Q). Note however that even with this rectriction, the problem (Q) may be singular (i.e. the viscosities $\nu+\frac{C_{\theta}}{\theta \phi}$ and $\nu+\frac{C_{\phi}}{\theta \phi}$ may be unbounded). Moreover, because we allow $\nu \equiv 0$ the equations may degenerate (i.e. the viscosities may vanish). Hence without additional restriction there may be various non equivalent notions of weak solution (see for instance \cite{dreyfuss}).\\ 

In fact a good compromise between respect of the physics, simplification of the mathematical study and obtention of significative results, is to restrict 
$\theta$ and $\phi$ to be within the classe $\mathcal S$ defined by:  
$$ \mathcal S = \big\{ f: \Omega \to \mathbb R^+ \text{ such that } f\in H^1(\Omega)\cap L^{\infty}(\Omega), \ f^{-1} \in  L^{\infty}(\Omega) \big\}. $$
In particular, if the parameters appearing in (Q) are sufficiently regular and if we restrict $\theta$ and $\phi$ to be within the classe $\mathcal S$, then the notion 
of a weak solution for (Q) is univocally defined: it is a distributional solution ($\theta,\phi$) that satisfies the boundary conditions in the sense of the trace. \\ 
In this last situation we will tell that ($\theta,\phi$) is a {\bf weak solution of (Q) in the class $\mathcal S$}. \\ 

In order to can consider such a weak solution for (Q) we shall precise in the following some sufficient conditions of regularity for 
the data. \\  
Let $N$=2 or $3$ denote the dimension of the domain $\Omega$, and $r$ be a fixed number such that:  
\begin{equation}
r>\frac{N}{2}.  \label{r}
\end{equation}
We then have the following continuous injection (see lemma \ref{lemrep}):  
\begin{equation}
L^r(\Omega) \subset  W^{-1,\beta}(\Omega), \quad \text{where }\beta = r^{*} > N. \label{beta}
\end{equation}

Recall that $D=\div({\bf u})$. We will consider the following assumptions: 
%\begin{boitenumeroteeavecunedoublebarre}{}
\begin{itemize}
\item Assumptions on $\Omega$:  
\begin{equation} 
 \Omega \subset \mathbb R^N \quad \text{is open, bounded and it has a Lipschitz boundary }
 \partial \Omega. \label{h-dom1}
\end{equation} 
\item Assumptions on the flow data \footnote{when $N=2$ one assumption in (\ref{h-u1}) can be relaxed: 
 ${\bf u}\in (L^{p}(\Omega))^{2}$ with $p>2$ (instead of $p=3$) is sufficient, but this would not improve any result significantly.}${\bf u}, F, D,\rho$ and $\nu$:  
\begin{eqnarray}
   &&\nu \geq 0, \label{h-nu} \\
    &&F,\rho:\ \Omega \to \mathbb R^+,\quad \rho,\rho^{-1} \in L^{\infty}(\Omega), \quad F \in L^{r}(\Omega), \label{h-rho} \\    	  
	  &&{\bf u} \in  (L^3(\Omega))^{N}, D \in L^r(\Omega), \div(\rho {\bf u})=0,\label{h-u1} \\ 
	  &&{\bf u}\cdot{\bf n} = 0 \text{ on } \partial \Omega \label{h-bord}.
\end{eqnarray}
\item Assumptions on the turbulent quantities on the boundary: 
\begin{eqnarray}	   
	  && a,b \in H^{1/2}(\partial \Omega)\cap L^{\infty}(\partial \Omega), \notag \\ 
	  && a(x),b(x) \geq \delta>0 \text{ a.e. }x \in \partial \Omega, \label{h-bord2}
\end{eqnarray}
where $\delta>0$ is a fixed number. 
\item Assumptions on the model coefficients: 
\begin{eqnarray} 
	  &&  C_{i}: \Omega \to \mathbb R^{+},  \ C_5,C_8 \in L^{r}(\Omega),  \text{ for } i\neq 5,8: C_i \in L^{\infty}(\Omega) \label{h-Ci1} \\
	  && C_{\text{ind}}:  \Omega\times(\mathbb R^{+})^2 \to\mathbb R^{+} \text{ are Caratheodory } \label{h-Ci2} \\ 
              && C_{\text{ind}}(x,v,w) \geq \alpha_{\text{ind}} >0 \quad \forall v,w \in \mathbb R^{+} \text{ and for a.a. } x\in \Omega \label{h-Ci3} \\ 
              && \forall v,w \in \mathbb R^{+}, \ x\to C_{\text{ind}}(x,v,w) \in L^{\infty}(\Omega), \label{h-Ci4}
\end{eqnarray}
where in (\ref{h-Ci2})-(\ref{h-Ci4}) $C_{\text{ind}}$ means $C_{\theta}$ or $C_{\phi}$. The assumption (\ref{h-Ci2}) signifies that $\forall v,w  \in \mathbb R^{+}$, 
$x\to C_{\text{ind}}(x,u,v)$ is measurable and for a.a. $x\in\Omega: (v,w)\to C_{\text{ind}}(x,v,w)$ is continuous. This ensures that $C_{\text{ind}}(x,\theta,\phi)$ 
is measurable when $\theta$ and $\phi$ are measurable. The condition (\ref{h-Ci3}) means that $C_{\text{ind}}$ is uniformly positive, whereas (\ref{h-Ci4}) tells that 
$C_{\text{ind}}(x,\theta,\phi)$ remains bounded if $\theta$ and $\phi$ are bounded.
\end{itemize}
%\end{boitenumeroteeavecunedoublebarre}

We will study problem (Q) under the {\bf main assumption}:  
$$(A_0): \quad (\ref{r})-(\ref{h-Ci4}) 
\text{ are satisfied.}$$
%\end{breakbox} \noindent \\

Note that in the main situation $\text{($A_0$)}$ the assumption (\ref{h-nu}) made for $\nu$ allows the possibilty $\nu\equiv 0$. In other words the molecular viscosity $\nu$ can 
be neglected in the model. This is often chosen in practice because the eddy viscosities $\frac{C_{\theta}}{\theta\phi}$ and $\frac{C_{\phi}}{\theta\phi}$ are dominant in the physical situations (see \cite{moha,moha2}). \\ 
Remark also that the coefficients $C_i$ are allowed to depend on $x$, and the viscosity parameters $C_{\theta},C_{\phi}$ may depend on $x,\theta,\phi$. \\ 

For a given function $f: \Omega \to \mathbb R$, we shall use the notations $f^{+}$ and $ f^{-}$ to represent the positive and negative parts of $f$, that is: 
$$f=f^{+}+f^{-}, \quad f^{+}(x)=\max(f(x),0) \geq 0, \quad  f^{-}(x)=\min(f(x),0) \leq 0.$$ 
We will also consider some assumption of low compressibility of the form: 
\begin{equation} 
\|D^{+}\|_{L^r(\Omega)} \leq \tau, \label{lowcomp}
\end{equation}   
for some $\tau>0$ that will be precised. \\ 
This last kind of condition seems to be necessary (see the Appendix) in order to obtain a weak solution for (Q) in the three dimensional case, whereas when $N=2$ we shall use some particularities of the situation to obtain a weak solution without any assumption of low compressibility. Nevertheless in this case we will assume that in addition to $\text{($A_0$)}$ the following condition is satisfied: 
$$\text{($A_1$)}: \quad \frac{C_4^2}{C_3} \in L^r, \quad c_6=0.014, \ c_7=0.104, \ c_8=0.84.$$
In this last condition the values for $c_6,c_7$ and $c_8$ are in fact their classical constant values (see (\ref{val4}) in the Appendix)

In the sequel we denote by $DATA$ some quantity depending only on the data under the assumption $(A_0)$, i.e.: 
\begin{align}
DATA = Const\big(&\Omega,a,b,\|\mathbf{u}\|_{(L^{3})^N},\alpha_{\theta},\alpha_{\phi},(\|C_j\|_{L^r})_{j=5,8},\|F\|_{L^r}, \notag \\ 
&(\|C_i\|_{L^{\infty}})_{i=3,4,6,7}\big). \label{datas}
\end{align}
Note that $DATA$ does not depends on $D$ and $\nu$. \\ 
The exact form of the dependency (i.e. the function $Const$) is allowed to change from one part of the text to another.
\subsection{Main results}

We shall establish two theorems of existence. The first one applies if  N=2 or 3 and the second one is limited to the case N=2. %\noindent \\  
\begin{theo}\label{theo1}
Assume that ($A_0$) holds. Then there exists $\tau > 0$ such that if  $\|D^{+}\|_{L^r(\Omega)} \leq \tau$ 
then problem (Q) admits at least one weak solution $(\theta,\phi)$ in the class $\mathcal S$. 
\end{theo}

\begin{theo}\label{theo2}
Assume that $N=2$ and ($A_0$),($A_1$) hold. Then problem (Q) admits at least one weak solution $(\theta,\phi)$  in the class $\mathcal S$.    
\end{theo}
In all the situations we have the following regularity result:
\begin{theo}\label{theo3}
Let $(\theta,\phi)$ be a weak solution of (Q) in the class $\mathcal S$ and assume that ($A_0$) is satisfied. We have:
\begin{itemize}
\item[i)] If ${\bf u}\in (L^\infty(\Omega))^{N}$ and $a,b$ are H\"{o}lder continuous, then $(\theta,\phi) \in (\mathcal{C}^{0,\alpha}(\overline{\Omega}))^2$, for some $\alpha > 0$
\item[ii)] Assume that in addition, the following conditions are satisfied:
\begin{align*}
&\partial \Omega,a \text{ and } b \text{ are of class } \mathcal{C}^{2,\alpha}, \quad  \rho {\bf u} \in (\mathcal{C}^{1,\alpha}(\overline{\Omega}))^N\\ 
&F \in \mathcal{C}^{0,\alpha}(\overline{\Omega}),\quad  \forall i=3,..,8: \ C_i \in \mathcal{C}^{0,\alpha}(\overline{\Omega}) \\ 
&C_{\text{ind}}(x,v,w) \in \mathcal{C}^{1,\alpha}(\overline{\Omega}\times(\mathbb R^{+})^2)
\end{align*} 
Then $(\theta,\phi) \in (\mathcal{C}^{2,\alpha}(\overline{\Omega}))^2$ and it is a classical solution of (Q).
\end{itemize}
\end{theo}

\subsection{Discussion on the results} \label{disc}

The major feature of our work is to treat the compressible situation in a general framework. The compressible case is interesting for several applications (see for instance \cite{moha-num,davi}). However the additional 
terms of the model arising from the compressibility effects induce important complications for its analysis. Roughly speaking: these additional terms are 
responsible of an increase or decrease (depending on the signe of $D$) of the turbulence. Then the balance between the increase/decrease of the source terms of turbulence and the possible explosion/vanishing of the turbulent viscosities is difficult to control. \\   
Compared to previous works (see \cite{Lew,lafitte,moha,wu}) our basic assumption ($A_0$) made in theorem \ref{theo1} is very general, and we remark that: 
\begin{itemize}
\item[i)] We do not impose a free divergence condition on ${\bf u}$ and the mean density $\rho$ is not supposed constant. Hence our analysis can handle compressible turbulent flows. The condition of the form (\ref{lowcomp}) assumed for $\div {\bf u}$ is much 
weaker than the corresponding assumption made in \cite{lafitte}. In dimension two, under the additional condition ($A_1$) we 
obtain in theorem \ref{theo2} an existence result without any additional assumption of low compressibility. The necessity of a condition of the form  (\ref{lowcomp}) when N=3 is discuted in the Appendix.
\item[ii)] We allow the viscosity parameters $C_{\theta},C_{\phi}$ to depend on $x,\theta,\phi$ whereas in the previous works these parameters were taken constant (except in \cite{marmol,marmol2} where however an artificial regularization was made, and only the incompressible situation was studied). 
\item[iii)] We only assume weak regularity for the mean flow, i.e: ${\bf u} \in (L^3(\Omega))^N$ and $\div {\bf u} \in L^r(\Omega)$, with some $r>3/2$, whereas in the previous works it was assumed ${\bf u} \in (L^{\infty}(\Omega))^N$ and $\div {\bf u} \in L^{\infty}(\Omega)$.
Our assumption is more interesting from a mathematical point of view because it is satisfied when ${\bf u}$ is a weak solution of the Navier Stokes 
equations. Hence our work could be used for a future analysis of the full coupled system Navier-Stokes plus (Q).
\item[iv)] All the coefficients $C_i$ are allowed to depend on $x\in\Omega$. Hence our study is also a step for the analysis of  
the Low-Reynolds $k-\epsilon$ model (in \cite{wu} the Low-Reynolds $k-\epsilon$ model is studied only in the incompressible and unsteady situation). Finally the boundary values for $\theta$ and $\varphi$ are not taken constants unlike 
the previous works.  
\end{itemize}

Note also that under the additional assumption ${\bf u} \in (L^{\infty}(\Omega))^N$ we give a H\"{o}lder continuity result for 
$k$ and $\theta$. Moreover, we establish an existence result for a classical solution under some smoothness assumptions 
on the data. These regularity results seem to be completly new. \\
 
Another feature of our work is to point out that the choice $\phi=\frac{\epsilon^2}{k^3}$ is indicated even in the compressible situation: this makes the dynamic of the $\phi$ equation stable when $N=2$ and stable under additional conditions when $N=3$ 
(see the Appendix). These results improve the study of the model presented in \cite{moha}, chap. 9. \\ 

In order to establish our theorems we establish intermediate results (see Proposition \ref{aux1}), concerning the existence, positivity  and regularity properties for a weak solution
of an elliptic scalar problem (possibly degenerate and singular) of the form: 
$$
\text{(S)}  \quad\left\{
	\begin{array}{l} 
	\rho{\bf u} \cdot \nabla \zeta -\div\left( (\nu+ \frac{\kappa}{\zeta}) \nabla \zeta\right) = g(x,\zeta) \quad \text{in } \Omega \\ 
     \zeta = \zeta_0 \quad \text{on } \partial \Omega, 
	\end{array}
	\right. \\  
$$
These results have also an independent mathematical interest.

\subsection{Organization of the paper}

$\bullet \ $In section 3 we shall recall some results concerning: the truncature at a fixed level and the Stampacchia's estimates. This 
last technique takes an important role in our analysis, moreover we shall need a precise control of the estimates. Hence we shall 
present it with some details and developments. \\ 
$\bullet \ $In section 4 we introduce a sequence ($Q_n$) of problems which approximate (Q). For $n$ fixed $(Q_n)$ is a PDE system of two scalar equations of the form (S): one equation for 
the unknown $\theta_{n+1}$ and one for $\phi_{n+1}$. The point is that the unknowns $\theta_{n+1}$ and $\phi_{n+1}$ are only weakly coupled. The coupling of the two equations is essentially realized through the quantities $\theta_{n}$ and $\phi_{n}$ calculated at the previous step. Hence we shall firstly study carefully the problem (S). The major tool used here are the Stampacchia's estimates. We next prove that ($Q_n$) is well posed. Hence we obtain an approximate sequence of solutions $(\theta_{n},\phi_{n})$ for problem (Q). 
Moreover, we prove that $\theta_{n}$ and $\phi_{n}$ are uniformly bounded from above and below, which are the key estimates.\noindent \\ 
$\bullet \ $ In section 5 we use the uniform bounds established in section 4, in order to extract a converging subsequence from $(\theta_{n},\phi_{n})$. We then 
prove that the limit $(\theta,\phi)$ is a weak solution of (Q) in the class $\mathcal{S}$. \\ 
Under the additional assumption ${\bf u} \in (L^{\infty}(\Omega))^3$ we are able to use the De Giorgi-Nash Theorem and we obtain an H\"older continuity result for $\theta,\phi$. 
By assuming in addition some smothness properties for the data we can iterate the Schauder estimates and prove Theorem \ref{theo3}. \\ 
$\bullet \ $In section 6 (Appendix) we present the derivation of the $\theta-\phi$ model from the $k-\epsilon$ one. Moreover we 
justify that the choice $\phi=\frac{\epsilon^2}{k^3}$ is valid even in the compressible situation. The justification uses in particular 
a property of positivity of the function $F$. We also dicuss briefly the necessity of the low compressibility assumption when $N=3$. 
Finally we recall a generalized version of the chain rule for $G(u)$ where $G$ is a Lipchitz function and $u$ a Sobolev one.   
\section{Mathematical background}
In this section we shall recall some results concerning: the truncature at a fixed level and the Stampacchia's estimates. This 
last technique takes an important role in our analysis, moreover we shall need a precise control of the estimates. Hence we shall 
present here the technique with some details and developments.
As in the rest on the paper we denote by  $\Omega \subset \mathbb R^N$ a bounded open Lipschitz domain. These properties for $\Omega$ are always implicity assumed if they are not precised. 

\subsection{Truncatures and related properties} \label{appendixI}
The technique of Stampacchia is based on the use of special test functions which are constructed by using some truncatures. We shall 
recall some basic properties of the truncatures used in the paper. An important tool is the generalized chain rule (see Theorem \ref{stam-lip} in the Annexe). \\ 
Let $l>0$ we denote by $T_l$ the truncature function $T_l: \mathbb
R \to \mathbb R$ defined by 
\begin{equation}
T_l(s) = \quad\left\{
	\begin{array}{l} 
	l \quad \text{if } s > l \\
        s \quad \text{if } -l\leq s\leq l \\ 
        -l \quad \text{if } s < -l 
	\end{array}
	\right. \label{tronc}
\end{equation} 
Let $v \in H^1(\Omega)$.
By applying Theorem \ref{stam-lip} we see that $T_l(v)\in H^1(\Omega)
\cap L^{\infty}(\Omega)$. Moreover, if we denote by $\Omega_{v,l}$ the
set $\{ x \in \Omega \text{ s.t. } |v(x)| \leq l \}$ then we have: 
$$\nabla T_l(v)=  \quad\left\{
	\begin{array}{l} 
	\nabla v \quad \text{in } \Omega_{v,l} \\
        0 \qquad \text{elsewhere} 
	\end{array}
	\right. 
$$
Note that  $T_l(.)$ truncates both the positive and the negative large values. In some cases we need only to 
truncate in one side. For this reason we introduce the semi-truncatures $T_{l,+}$ and  $T_{l,-}$ defined by: 
$$  T_{l,+}(t)= \quad\left\{
	\begin{array}{l} 
        l \quad \text{if } t >  l \\ 
        t \quad \text{elsewhere }  
	\end{array} \right. \quad 
     T_{l,-}(t)= \quad\left\{
	\begin{array}{l} 
        -l \quad \text{if } t < -l \\ 
        t \quad \text{elsewhere }  
	\end{array} \right.
$$
We then have the decomposition: $T_l=T_{l,+}\circ T_{l,-} = T_{l,-}\circ T_{l,+}$.  \\ 
For given $v_l\in H^1(\Omega)$ and $s>0$, we shall also consider  
\begin{equation} 
\psi_s(v) = v_l - T_s(v_l) =  \quad\left\{
	\begin{array}{l} 
	 v_l - s \quad \text{in } A_{s,l}^+ \\
   v_l + s \quad \text{in }  A_{s,l}^- \\ 
         0  \quad \text{elsewhere}
	\end{array}
	\right. \label{psis}
\end{equation} 
where we have used the notations: 
\begin{equation} 
A_{s,l}^+:=\{v_l\geq  s\}, \ A_{s,l}^-:=\{v_l\leq  -s\} \text{ and } A_{s,l}= A_{s,l}^+ \cup A_{s,l}^-. \label{Asl}
\end{equation}
Let also  $\psi_{s,\pm}$ be the functions defined above (\ref{psis}) while replacing $T_s$ by $T_{s,+}$ or by $T_{s,-}$ . It is 
easy to verify that $\psi_{s,+}$ (resp. $\psi_{s,+}$) is in fact the positive (resp. the negative) part of $\psi_s$. In other 
words, we have :
\begin{align*} 
\psi_{s,+}(v_l) &= \psi_{s}^{+}(v_l) = (v_l - s) {\mathbf 1}_{A_{s,l}^+} \geq 0, 
\quad \psi_{s,-}(v_l) = \psi_{s}^{-}(v_l) =(v_l + s) {\mathbf 1}_{A_{s,l}^-} \leq 0. \\
\psi_{s}  &=  \psi_s^{+} + \psi_{s}^{-}.
\end{align*}
The function $\psi_s$ has the following properties:
\begin{lemm} \label{lem123} \noindent
\begin{itemize}
\item[i)] $\psi_s \in H^1(\Omega)$ and $\nabla \psi_s = {\mathbf 1}_{A_{s,l}} \nabla v_l$
\item[ii)] if $\gamma v_l \in L^{\infty}(\partial \Omega)$ then for all $s>\|\gamma v_l\|_{L^{\infty}(\partial \Omega)}$ we have 
$\psi_s \in H_0^1(\Omega),$
\end{itemize}
where $\gamma: H^1(\Omega) \to H^{1/2}(\partial \Omega)$ denotes the trace function.
\end{lemm}
\begin{proof}
Point i) is a direct consequence of Theorem \ref{stam-lip}. Property ii) is proved in \cite{lady}, lemma 3.3 p.53.
\end{proof} 
\subsection{The Stampacchia estimates} \label{appendixII} 

The Stampacchia estimates is a general method which allows one to obtain an $L^\infty$-estimate for the weak solution of a large class 
of elliptic PDEs of the second order. The $L^\infty$-estimate presented in the original work \cite{stam} depends on various 
quantities related to the PDE problem studied, but the exact dependency is not established. 
In our analysis we need a precise control of the $L^\infty$-estimates with respect 
to some quantities (in particular with respect to the diffusion coefficient of the PDEs). Hence in the following we take over and detail  
the technique  in order to obtain a more precise $L^\infty$-estimates.  \\

The Stampacchia estimates are established by using the test functions $\psi_s$ (or $\psi_s^\pm$) defined previously, 
where in this case $v$ (resp. $v_l$) is a weak solution of the problem (resp. the sequence of problems) considered.  \\ 
For technical reasons we need a classical result concerning some relationship between $L^r$ functions and linear form on 
Sobolev spaces: 
\begin{lemm} \label{lemrep}
Let $1\leq r <\infty$ and $E \in L^r(\Omega)$. Then 
\begin{equation} \label{propE}
E \in W^{-1,\beta}(\Omega), \ \text{with } \beta=r^* = \frac{Nr}{N-r},
\end{equation}
and there exists $\tilde{E} \in (L^{\beta}(\Omega))^3$ such that: 
\begin{equation} 
\int_{\Omega} E\varphi = \int_{\Omega} \tilde{E}\cdot\nabla\varphi\quad \forall \varphi\in\mathcal{D}(\Omega), \quad \|\tilde{E}\|_{(L^{\beta}(\Omega))^3} \leq  \|E\|_{L^r(\Omega)}. \label{repE}
\end{equation}
Moreover we have: 
\begin{equation} \label{propE2}
r > \frac{N}{2} \Rightarrow \beta > N.
\end{equation}
\end{lemm}
\begin{proof} \noindent  
Property (\ref{propE}) is easy to prove by using the Sobolev injection Theorem together with the H\"{o}lder inequality:  
$\varphi \to \int_{\Omega} E \varphi$ is a linear form on $W_0^{1,p}$ if $p^*:=\frac{Np}{N-p} \geq \frac{r}{r-1}$. This last condition holds for 
$p=P:=\frac{Nr}{r(N+1)-N}$. Hence $P'=p/(p-1)=\frac{Nr}{N-r}=r^*$. \\ 
We next obtain (\ref{repE}) by using a classical result (see \cite{brezis} Proposition IX.20). \\ 
Finally, if we assume that $r > \frac{N}{2}$ then  $\beta=r^* >\frac{N^2/2}{N-N/2} = N$. 
\end{proof}

The Stampacchia technique works in two steps: the first one is dependent of the problem (or the sequence of problems) studied and the second one is independent of it. 
Here the purpose is to present the key ingredients of these two steps. Because the first one is dependent on the problem studied we 
cannot present it here in its enterity, but we will consider a simple problem which contains the main technical points 
(in fact this introductive presentation will be useful to treat a more complicated class of problems in Section 4). 
Let $(v_l) \subset H^1_0(\Omega)$ be a sequence of functions satisfying: 
\begin{equation}
\int_{\Omega} \nu_l \nabla v_l \nabla \phi = \int_{\Omega} g_l \phi, \quad \forall \phi \in H^1_0, \label{pbVl}
\end{equation}
where $(\nu_l)$ is a given sequence of strictly positive bounded functions and $(g_l) \subset L^r(\Omega)$, with $r>\frac{N}{2}$. 
Let also $m_l,M_l$ denote the bounds from above and below for $\nu_l$, that is:
$$0 < m_l \leq \nu_l \leq M_l < \infty, \quad \text{a.e. in } \Omega.$$
-\textit{Step 1} \noindent \\
By testing (\ref{pbVl}) with $\psi_s=\psi_s(v_l)$, we obtain:
\begin{equation}
m_l \int_{\Omega} |\nabla \psi_s|^2 = \int_{\Omega} g_l \psi_s \overset{\text{by lemma }\ref{lemrep}}{=} 
\int_{\Omega} E_l \nabla \psi_s, \label{step1.1}
\end{equation}
with  $E_l \in L^{\beta}, \|E_l\|_{L^\beta} \leq \|g_l\|_{L^r}$ for some $\beta=\beta(r) > N$. \\ 
Recall also that $\text{Supp }\psi_s(v_l) \subset A_{l,s}$. Hence by using the H\"{o}lder inequality we obtain:
\begin{align*} 
\int_{\Omega} E_l \nabla \psi_s &\leq \|\nabla \psi_s\|_{L^2} (\int_{A_{l,s}}|E_l|^2)^{\frac{1}{2}} 
\leq  \|\nabla \psi_s\|_{L^2} \|E_l\|_{L^\beta} (\int_{A_{l,s}} 1)^{\frac{1}{2(\frac{\beta}{2})'}} \\ 
&\leq \|\nabla \psi_s\|_{L^2} \|E_l\|_{L^\beta} |A_{l,s}|^{\frac{\beta-2}{2\beta}}.
\end{align*}
Consequently (\ref{step1.1}) leads to: 
\begin{equation}
\|\psi_s\|^2_{H^1_0(\Omega)^2} \leq C |A_{l,s}|^{\Phi}, \quad \text{whith } C>0, \Phi>\frac{2}{2^{*}}=\frac{N-2}{N}.\label{keyStam} 
\end{equation}
{\bf This is the key estimate} needed to pass at the second step which is independent of the problem studied. \noindent \\ 
Note that with the particular sequence of problems (\ref{pbVl})  chosen here, the constants $C$ and $\Phi$ are: 
\begin{equation}
C=\frac{\|g_l\|_{L^r}^2}{m_l^2}, \quad \Phi=\frac{\beta-2}{\beta}. \label{exemple}
\end{equation}
Hence $\Phi$ does not depend on $l,s$. Moreover if we assume that ($g_l$) is uniformly bounded in the $L^r$-norm and 
that ($m_l$) is uniformly bounded from above by a strictly positive constant, then $C$ is also independent on $l,s$. \\ 
This is an important point because as we will see it hereafter, an estimate (\ref{keyStam}) with $C$ and $\Phi$ independent on 
$l,s$ leads to a uniform $L^{\infty}$ bound for ($v_l$).\\ 
-\textit{Step 2} \noindent \\ 
Assume that we have obtained (\ref{keyStam}). We can obtain an $L^{\infty}$-estimate for $v_l$ as follows. \\  
Let $2^{*}=\frac{2N}{N-2}$ be the Sobolev exponent associated to 2 in dimension $N$. 
By using the Poincar\'e-Sobolev inequality we have: 
\begin{equation}
\big( \int_{A_{l,s}} |\psi_s|^{2^{*}}\big)^{2/2^{*}} \leq C_1 
\|\psi_s\|_{H^1_0(\Omega)}^2, \quad  C_1=C_1(|\Omega|) \label{aII.1}
\end{equation}
Let now $t > s$. It is clear that $A_{l,t} \subset A_{l,s}$ and
consequently 
\begin{equation}
\big( \int_{A_{l,s}} |\psi_s|^{2^{*}}\big)^{2/2^{*}} \geq 
\big( \int_{A_{l,t}} |\psi_s|^{2^{*}}\big)^{2/2^{*}} \geq 
\big( \int_{A_{l,t}} |t-s|^{2^{*}}\big)^{2/2^{*}} \geq 
|t-s|^2 \big|A_{l,t}\big|^{2/2^{*}} \label{aII.2}
\end{equation}
We set 
\begin{equation} 
\chi(s) := \big|A_{l,s}\big|, \quad \forall s \geq 0. \label{defChi}
\end{equation}
For fixed $l$, $\chi$ is a decreasing function, and from the
estimates (\ref{keyStam})-(\ref{aII.2}), we obtain 
$$\chi(t) \leq C_2 |\chi(s)|^{2^{*}\Phi/2}(t-s)^{-2^{*}} \quad 
\forall t > s \geq 0, \quad C_2=(C_1 C)^{2^{*}/2}.$$
Recall that we have assumed in (\ref{keyStam}) that $\Phi>\frac{2}{2^{*}}$. Hence $2^{*}\Phi/2>1$ and by using Lemma 4.1 in \cite{stam} we obtain: 
\begin{equation} 
\chi(d) = 0, \quad d=2^\frac{2^{*}\Phi}{2^{*}\Phi-2} C_2 ^{1/2^*}|\Omega|^{\frac{\Phi}{2}-\frac{1}{2^*}} < \infty. \label{borneStam} 
\end{equation} 
This property tells exactly that: 
\begin{equation}
\|v_l\|_{L^{\infty}(\Omega)} \leq d, \quad d=C_3(\Phi,N,|\Omega|) \sqrt{C}. \label{conc}
\end{equation} 
In particular $d$ does not depend on $l$ if the constants $C$ and $\Phi$ appearing in (\ref{keyStam}) 
are independent of $l$. For instance with the particular sequence of problems (\ref{pbVl}) the constants $C$ and $\Phi$ are 
given by (\ref{exemple}), and if we assume that $\|g_l\|_{L^r}\leq K, \ m_l \geq m>0$ we obtain:
$$d=\tilde{C} \frac{K}{m}, \quad \tilde{C}=\tilde{C}(\Omega,N,r).$$
\begin{rem} \label{remStam} \noindent 
\begin{itemize}
\item If you are only interested in obtaining an uniform majoration or minoration for $v_l$ then instead of (\ref{keyStam}) it 
suffices to have
\begin{equation}
\|\psi_s^{\pm}\|^2_{H^1_0(\Omega)^2} \leq C^{\pm} |A_{l,s}^{\pm}|^{\Phi^{\pm}}, \quad \text{with } C^{\pm}>0, \Phi^{\pm}>\frac{N-2}{N}. \label{keybis}
\end{equation}
In fact in this case we consider $\chi_{\pm}(s) := \big|A_{l,s}^{\pm}\big|$ instead of (\ref{defChi}). This function is decreasing and 
we obtain again (\ref{borneStam}). But now this property tells: 
$$\pm v_l(x) \leq d^{\pm} \text{ a.e. in }\Omega, \quad d^{\pm}=\tilde{C}^{\pm}(\Phi^{\pm},N,|\Omega|)\sqrt{C^{\pm}}.$$
\item Let again ($v_l$) be a sequence of functions satisfying (\ref{pbVl}) and assume that $m_l\geq m >0$. Then we have: 
\begin{align}
g_l &\leq h_{l}  \text{ and } H:=\sup_{l} \|h_l\|_{L^r}<\infty  \Rightarrow  v_l \leq d_1 \text{ a.e. in } \Omega, \label{majo}\\ 
g_l &\geq k_{l}  \text{ and } K:=\sup_{l}\|k_l\|_{L^r} <\infty \Rightarrow  v_l \geq -d_2 \text{ a.e. in } \Omega, \label{mino} \\ 
&\text{with }d_1,d_2 >0, \ d_1=\tilde{C}\frac{H}{m}, \ d_2=\tilde{C}\frac{K}{m}, \quad \tilde{C}=\tilde{C}(\Phi,N,|\Omega|). \notag 
\end{align}
The proof of (\ref{majo}) and (\ref{mino}) are obtained by taking over the first step of the technique of Stampacchia, but 
by using $\psi_s^{\pm}$ as test function instead of $\psi_s$. \\ 
In fact, the function $\psi_s^{+}$ is positive. Hence instead of (\ref{step1.1}) we have: 
 $$
m \int_{\Omega} |\nabla \psi_s^{+}|^2 = \int_{\Omega} g_l \psi_s^{+} \leq \int_{\Omega} h_l \psi_s^{+} = 
\int_{\Omega} E_{l,+} \nabla \psi_s^+,
$$
with  $E_{l,+} \in L^{\beta}, \|E_{l,+}\|_{L^\beta} \leq \|h_l\|_{L^r}\leq H$. Consequently, 
in this case we obtain  (\ref{keybis}) for the function $\psi_s^+$ , with $C^{+}=\frac{H^2}{m^2}$, and the uniform majoration (\ref{majo}) follows. \\ 
The relation (\ref{mino}) can be proven by using $\psi_s^{-}$ as test function in (\ref{pbVl}). In fact we remark now  
that $\psi_s^{-}$ is negative. Hence instead of  (\ref{step1.1}) we obtain:
 $$
m \int_{\Omega} |\nabla \psi_s^{-}|^2 = \int_{\Omega} g_l \psi_s^{-} \leq \int_{\Omega} k_l \psi_s^{-} = 
\int_{\Omega} E_{l,-} \nabla \psi_s^-,
$$ 
with  $E_{l,-} \in L^{\beta}, \|E_{l,-}\|_{L^\beta} \leq \|k_l\|_{L^r} \leq K$. Consequently we now obtain  (\ref{keybis}) for the function $\psi_s^-$, with $C^{-}=\frac{K^2}{m^2}$. This implies: $-v_l \leq d_2$ and consequently $v_l \geq -d_2$.
\end{itemize}
\end{rem}  
\section{Approximate sequence and estimates}

Let $g\in H^{1/2}(\partial\Omega)$, we denote by $\mathcal R g$ its harmonic lifting, that is: 
$$\mathcal R g \in H^1(\Omega), \mathcal R g = g \text{ on  }\partial\Omega \ \text{  and  } \Delta R g = 0 \text{ on }\Omega.$$
We define the functions $\theta_0$ and $\phi_0$ by the formula:
\begin{equation}
\theta_0:=\mathcal R a \quad \phi_0:=\mathcal R b. \label{etape1}
\end{equation}
Hence, by the using the maximum principle (see \cite{brezis} p.189 and \cite{galia}) together with the condition (\ref{h-bord2}) we obtain: 
\begin{equation}
0< \delta \leq \theta_0 \leq \|a\|_{L^{\infty}(\partial \Omega)}, \quad \delta \leq \phi_0 \leq \|b\|_{L^{\infty}(\partial \Omega)}. \label{etape1-1}
\end{equation} 

Let now $n\geq 0$, ($\theta_n,\phi_n$) be given and 
$$
C_{\theta}^{(n)}(.):=C_{\theta}(.,\theta_n(.),\phi_n(.)), \quad  C_{\phi}^{(n)}(.):=C_{\phi}(.,\theta_n(.),\phi_n(.)).
$$ 
In order to construct an approximate solution $(\theta_{n+1}, \phi_{n+1})$ for problem (Q), we introduce the following system:
 $$
\text{($Q_n$)}  \quad\left\{
	\begin{array}{l} 
	\rho{\bf u} \cdot  \nabla \theta_{n+1} -\div\left( (\nu + \frac{C_{\theta}^{(n)}}{\theta_{n+1} \phi_n}) \nabla \theta_{n+1}\right) = g_{\theta}^{(n)}\quad \text{in } \Omega, \\ 
     \rho{\bf u} \cdot \nabla \phi_{n+1} -\div\big( (\nu+\frac{C_{\phi}^{(n)}}{\theta_{n+1}\phi_{n+1}}) \nabla \phi_{n+1}\big)
= g_{\phi}^{(n)} \quad \text{in }\Omega ,\\ 
	 \theta_{n+1} = a \text{ and } \phi_{n+1} = b \quad \text{on } \partial \Omega,
	\end{array}
	\right. 
$$  
where we used the notations: 
$$g_{\theta}^{(n)}:=C_5- C_3 F \theta_{n+1}^2+C_4 \theta_{n+1}D, \quad 
g_{\phi}^{(n)}:=-\phi_n(C_8\theta_{n+1}^{-1}+C_6\theta_{n+1} F-C_7 D).
$$ 
For $n\in\mathbb N$ we denote by ($H_n$) the following condition: 
$$
\text{($H_n$)}  \quad\left\{
	\begin{array}{l}
	  \theta_n,\phi_n,\theta_n^{-1},\phi_n^{-1} \in L^{\infty}(\Omega), \\ 
	  \theta_n,\phi_n \geq 0.
	  \end{array}
	\right.
$$
Let $\varphi_{\text{max}}$ be a fixed real number such that: 
\begin{equation}
\varphi_{\text{max}}>\|b\|_{L^{\infty}(\partial \Omega)}. \label{phimax}
\end{equation}
We shall also consider the condition: 
$\text{($K_n$)} := \text{($H_n$)} \ + \  (\phi_n \leq \varphi_{\text{max}}).$

Note that (\ref{etape1-1}) shows that the condition ($K_n$) is statisfied for $n=0$. We will
prove in the sequel that under condition ($K_n$) we can obtain a weak solution
$(\theta_{n+1}, \phi_{n+1})$ for problem ($Q_n$), with moreover
$(\theta_{n+1}, \phi_{n+1})$ satisfying the condition ($K_{n+1}$). This last property 
ensures the right definition of an approximate sequence. More precisely, we have:  
%\begin{breakbox}
\begin{pro} \label{pro1} 
Let $n\in\mathbb N$ be given and assume that ($A_0$) is satisfied.
%\begin{itemize}
Let  also ($\theta_n,\phi_n$) be given and satisfying condition ($K_n$). There exists $\tau > 0$ depending only on $DATA$ such that if $\| D^{+} \|_{L^r(\Omega)} \leq \tau$
then problem ($Q_n$) admits at least one weak solution $(\theta_{n+1}, \phi_{n+1}) \subset  \big(H^1(\Omega)\times L^{\infty}(\Omega) \big)^2 $. \\ 
Moreover $(\theta_{n+1}, \phi_{n+1})$ satifies condition ($K_{n+1}$) and the estimates: 
\begin{eqnarray}
0 < \theta_{\text{min}} \leq &\theta_{n+1} & \leq \theta_{\text{max}}, \label{est1_prop} \\ 
0 < \phi_{\text{min}} \leq &\varphi_{n+1}&\leq \varphi_{\text{max}}, \label{est2_prop}
\end{eqnarray}
where $\varphi_{\text{max}}$ was fixed in (\ref{phimax}) and $\phi_{\text{min}}, \theta_{\text{min}}$, $\theta_{\text{max}}$ are positive 
numbers depending on $DATA$, but not on $n$.
\end{pro}
\begin{rem}
Proposition \ref{pro1} is the key result that will be used later on to prove Theorem \ref{theo1}, whereas for Theorem \ref{theo2} we shall establish and use a more simple version of this proposition (see Subsection \ref{theo2D}).
\end{rem}
In order to prove the proposition we establish intermediate results.
\subsection{Auxiliary results}

Let $n\in \mathbb N$ and ($\theta_n,\phi_n$) be given and satisfying ($K_n$). We want to obtain ($\theta_{n+1},\phi_{n+1}$) 
by solving ($Q_n$) and in order to iterate the algorithm we also want that  ($ \theta_{n+1},\phi_{n+1}$) satisfies 
($K_{n+1}$). \\ 
Remark that the system ($Q_n$) is composed of two coupled scalar elliptic equation in divergence form, with a 
possible singular and degenerate structure. Hence the goal of this subsection is to study this last kind of scalar problem. \\ 

In order to do this, we first introduce a weight $\kappa: \mathbb R^+ \to  \mathbb R^+$ which is assumed to be mesurable and satisfying: 
\begin{equation} 
0 < \kappa_0 \leq \kappa  \leq \kappa_1 \quad \text{a.e. in } \Omega, \label{h-rho}
\end{equation}
where $\kappa_0$ and $\kappa_1$ are two given reals. \\ 
Let also $g: \Omega\times \mathbb R^+ \to \mathbb R$ be a Caratheodory function (i.e. $\forall u \in \mathbb R^+: $ 
$x\to g(x,u)$ is measurable, and for a.a. $x\in\Omega: \ u\to g(x,u) $ is continuous).

Let us consider the scalar problem:
$$
\text{(S)}  \quad\left\{
	\begin{array}{l} 
	\rho{\bf u} \cdot \nabla \zeta -\div\left( (\nu+ \frac{\kappa}{\zeta}) \nabla \zeta\right) = g(x,\zeta) \quad \text{in } \Omega \\ 
     \zeta = \zeta_0 \quad \text{on } \partial \Omega, 
	\end{array}
	\right. \\  
$$
where $\zeta_0 \in H^{1/2}(\partial \Omega)\cap L^{\infty}(\partial \Omega),\ \zeta_0\geq \delta >0 $ a.e. in $\delta \Omega$, is given. We always assume that $\rho,\mathbf{u},\nu,\Omega$ which appear in (S) satisfy their corresponding conditions in ($A_0$).

Recall that we allow $\nu \equiv 0$ in ($A_0$). Hence problem (S) may degenerate (i.e. the viscosity $\nu+ \frac{\kappa}{\zeta}$ may vanish) when $\zeta \to \infty$. Moreover (S) is singular (i.e. the viscosity tends to infinity) when $\zeta \to 0$.  
	
We want now to find sufficient additional conditions for $g$ that guarantee the existence of a bounded positive weak solution 
for problem (S). Hence we shall consider:
\begin{align}
\gamma(x)&:=\sup_{u\in[0,1]}|g^{-}(x,u)| \in L^{r}(\Omega), \label{G2} \\ 
g^{+}(x,u) &\leq \gamma_1(x) + \gamma_2(x) h(u), \label{G3}
\end{align}
 where $\gamma_1,\gamma_2 \in L^{r}(\Omega)$ and  $h: \mathbb R^{+} \to\mathbb R^{+}$ is continuous. 
In fact, more than establishing only the existence of a bounded positive solution for (S), we are interested in obtaining  
some uniform (with respect to $\kappa_1$) bounds from above and below and some regularity results. We have:  
\begin{pro} \label{aux1} \noindent 
\begin{itemize}
\item[i)] 
Let $\kappa$ satisfying (\ref{h-rho}) and $g: \Omega\times \mathbb R^+ \to \mathbb R$ be a Caratheodory function satisfying 
(\ref{G2}), (\ref{G3}). There exists a real $\tau > 0$ depending on $\kappa_0, h$ such that if $\|\gamma_2 \|_{L^r(\Omega)} \leq \tau$  then there exists a weak solution $\zeta \in
H^1(\Omega)\cap L^{\infty}(\Omega)$ for problem (S). 
Moreover we have:
\begin{equation}  
0 < \zeta_{\text{min}} \leq \zeta \leq \zeta_{\text{max}}, \quad \zeta_{\text{min}}=e^{-\frac{C}{\kappa_0}}, \zeta_{\text{max}}=e^{\frac{C}{\kappa_0}}, \label{Linfty-1}
\end{equation}
where $C$ depends only on $\gamma,\gamma_1,\Omega,r,N,\zeta_0$. In particular $\zeta_{\text{min}}$ and $\zeta_{\text{max}}$ are independent of  $\kappa_1$. \\ 
In addition, the following extended \footnote{when $g^+ \equiv 0$ it is a maximum principle} maximum principle holds: 
\begin{equation}
 \|\zeta\|_{L^{\infty}(\Omega)} \leq \|\zeta_0\|_{L^{\infty}(\partial \Omega)} +  \tilde{C} \frac{\|g^+\|_{L^r}}{\kappa_0}\|\zeta\|_{L^{\infty}(\Omega)}, \ 
\tilde{C}=\tilde{C}(\Omega,N,r,\gamma_1). \label{Linfty-2}
\end{equation}
\item[ii)] Assume that in addition ${\bf u}\in (L^{\infty}(\Omega))^N$ and $\zeta_0$ is H\"{o}lder continuous. Then 
$\zeta \in \mathcal{C}^{0,\alpha}(\overline{\Omega})$ for some $0<\alpha<1$. Moreover if $\partial \Omega$ is of class 
$\mathcal{C}^{2,\alpha}, \ g\in\mathcal{C}^{0,\alpha}(\overline{\Omega}\times\mathbb R^{+}), \rho{\bf u}\in 
(\mathcal{C}^{1,\alpha}(\overline{\Omega}))^N, \kappa \in \mathcal{C}^{1,\alpha}(\overline{\Omega})$ and $\zeta_0\in\mathcal{C}^{2,\alpha}(\partial \Omega)$, then 
$\zeta \in \mathcal{C}^{2,\alpha}(\overline{\Omega})$ and it is a classical solution of (S).
\end{itemize}
\end{pro}
Before proving Proposition \ref{aux1} we establish an intermediate result. In a first step we consider the change of variable $v=\ln \zeta$  in (S), and 
for $l\in \mathbb N$ we introduce a truncated version ($S_l$) of the system obtained:  
$$\text{($S_l$)}  \quad\left\{
	\begin{array}{l} 
        \rho{\bf u}\nabla e^{T_l(v)} - \div((\nu e^{T_l(v)} + \kappa) \nabla v) = g(x,e^{T_l(v)})\quad \text{in } \Omega \\ 
        v = \ln \zeta_0 \quad \text{on } \partial \Omega 
	\end{array}
	\right. 
$$
We then establish:  
\begin{lemm} \label{lem0} \noindent  
\begin{itemize}
\item[i)]  
Let  $\kappa$ satisfying (\ref{h-rho}) and $g: \Omega\times \mathbb R^+\to\mathbb R$ be a Caratheodory function satisfying 
(\ref{G2}), (\ref{G3}).    
Then, for any $l\in \mathbb N$ there exists a weak solution $v=v_l \in H^1(\Omega)\cap L^{\infty}$ for the problem ($S_l$). 
\item[ii)] 
Let $(v_l)$ be the sequence given in i) and $l\geq 1$ be a fixed integer. Then there exists $\tau_l>0$ such that if the 
function $\gamma_2$ in (\ref{G2}) satisfies  $\|\gamma_2\|_{L^r} \leq \tau_l$ then we have:
\begin{equation}
\| v_l \|_{L^{\infty}(\Omega)} \leq C \frac{\|\gamma+\gamma_1\|_{L^r}}{\kappa_0}, \quad C=C(DATA,\zeta_0). \label{bornevl}
\end{equation}
In particular $C$ is independent of $\kappa, \nu$ and $l$. 
\end{itemize}
\end{lemm}
\begin{proof} 
\begin{itemize}
\item[i)] 
By using the divergence formula we obtain, for all $w\in H^1(\Omega)$: 
\begin{eqnarray*}
\int_{\Omega} \rho{\bf u} \nabla  e^{T_l(v)} . w &= - \int_{\Omega} e^{T_l(v)}
\div (\rho{\bf u}w) + \underbrace{\int_{\partial \Omega}  e^{T_l(v)} \rho w {\bf u} \cdot {\bf n}
  d\sigma}_{=0} \\ 
&= - \int_{\Omega} e^{T_l(v)} \rho {\bf u} \cdot \nabla w - \int_{\Omega}
e^{T_l(v)} \underbrace{\div (\rho{\bf u})}_{=0} w.
\end{eqnarray*}
Let $v_0:=\ln (\mathcal R \zeta_0)$ and consider the change of variable $\tilde{v}:=v-v_0$. Then problem ($S_l$) is equivalent to find $\tilde{v}\in H^1_0(\Omega)$ such that:
\begin{equation}
-\div \sigma(x,\tilde{v},\nabla \tilde{v}) = f(x,\tilde{v}) \quad \text{in } \mathcal{D}'(\Omega), \label{quasi}
\end{equation}
where $\sigma: \Omega \times \mathbb R \times \mathbb R^N \to \mathbb R^N$ and 
$f: \Omega \times \mathbb R \to \mathbb R$ are defined by: 
\begin{align*}
\sigma(x,w,{\bf G}) &= (\nu e^{T_l(w+v_0)}+\kappa(x)){\bf G} - \rho(x){\bf u}(x) e^{T_l(w+v_0(x))}, \\ 
f(x,w) &= g(x,e^{T_l(w+v_0(x))}).
\end{align*}
We now remark that (\ref{quasi}) is a quasilinear equation in divergence form. Moreover, it is easy to see that $f$ and $\sigma$ satisfy the classical growth assumptions and $\sigma$ satisfies also the classical coercivity condition. Note that: 
\begin{align*}
\langle\sigma(x,w,{\bf G})-\sigma(x,w,{\bf G'}),{\bf G}-{\bf G'}\rangle &= (\nu e^{T_l(w+v_0)}+\kappa)|{\bf G}- {\bf G'}|^2 \\ 
&\geq \kappa_0 |{\bf G}- {\bf G'}|^2.
\end{align*}
Hence $\sigma$ is strictly monotone in the third variable. We then conclude (see for instance \cite{augs} Theorem 1.5, or \cite{lady} Theorem 8.8 page 311) that 
there exists a weak solution $\tilde{v}\in H^1_0(\Omega)$ for (\ref{quasi}). \\ 
Consequently $v_l:=\tilde{v}+v_0$ is a weak solution for $(S_l)$, that is $\forall w 
\in H^1_0(\Omega)$: 
\begin{equation}
\int_{\Omega} (\nu e^{T_l(v_l)}+\kappa) \nabla v_l \nabla w + 
\int_{\Omega} \rho {\bf u} \nabla e^{T_l(v_l)} w
  = \int_{\Omega} g(x,e^{T_l(v_l)}) w. \label{ff1}
\end{equation}
By applying Theorem 4.2 page 108 in \cite{stam} we obtain: $v_l \in L^{\infty}(\Omega)$. 
\item[ii)]
\begin{equation}
\text{Let }\tau_l:=\frac{1}{h(e^l)} > 0 \text{ and assume now that } \|\gamma_2\|_{L^r} \leq \tau_l. \label{assum1}
\end{equation} 
With this additional assumption, we are able to obtain a useful estimation for 
$\|v_l\|_{L^\infty}$.Technically we will detail a method due to Stampacchia (see Subsection \ref{appendixII} for the notations and for an introduction of the method. Here only the first step of the technique will be developed further). 
Let $s>|\ln \|\zeta_0\|_{L^{\infty}(\partial \Omega)}|$, we consider the function $\psi_s = v_l - T_s(v_l)$. 
We have (see Lemma \ref{lem123}) $\psi_s \in H^1_0(\Omega)\cap L^{\infty}(\Omega)$, and by testing (\ref{ff1}) with $\psi_s$ we obtain: 
\begin{equation} 
\int_{\Omega} (\nu e^{T_l(v_l)}+\kappa) | \nabla \psi_s |^2 + \underbrace{\int_{\Omega} \rho {\bf u} \cdot
\nabla e^{T_l(v_l)} \psi_s}_{=:I} = \underbrace{\int_{\Omega}  g(x,e^{T_l(v_l)}) \psi_s}_{=:II}. \label{xx} 
\end{equation} 
$\bullet~$We will now evaluate the terms I and II. \\ 
The term I is simplified by writing one of its integrand factors, namely $\nabla e^{T_l(v_l)} \psi_s$ as a gradient. 
More precisely we have $\nabla e^{T_l(v_l)} \psi_s=\nabla \zeta_l$, with $\zeta_l \in H^{1}(\Omega)\cap L^{\infty}$ (see Lemma \ref{lemaI} in the Appendix). 
Hence by applying the divergence formula we see that I vanishes: 
\begin{equation}
I = \int_{\Omega} \rho {\bf u} \cdot \nabla \zeta_l 
\overset{\text{div. formula}}{=}-\int_{\Omega} \zeta_l \underbrace{\div {\rho \bf u}}_{=0} + \underbrace{\int_{\partial \Omega} \rho \zeta_l {\bf u} \cdot {\bf n} d\sigma}_{=0 \text{ by } 
(\ref{h-bord})}=0. \label{estI}
\end{equation}
We next estimate the term II: 
$$
II = \int_{\Omega}  g(x,e^{T_l(v_l)})\psi_s \leq  \underbrace{\int_{A_{s,l}^+}  g^{+}(x,e^{T_l(v_l)})\psi_s}_{:=II_1} + 
\underbrace{\int_{A_{s,l}^-}  g^{-}(x,e^{T_l(v_l)})\psi_s}_{:=II_2}.
$$
Remark now that on $A_{s,l}^-$ we have $v_l \leq -s \leq 0$, which implies that $e^{T_l(v_l)} \leq 1$. Consequently by using the 
assumption (\ref{G2}) we obtain: 
\begin{equation} 
II_2 \leq \int_{\Omega} \gamma \psi_s {\mathbf 1}_{A_{s,l}^-}. \label{estII}
\end{equation}
The term $II_1$ is majorated as follows: 
\begin{equation}
II_1 \underset{\text{by }(\ref{G3})}\leq \int_{A_{s,l}^+} (\gamma_1 + \gamma_2 h(e^l)) \psi_s 
\underset{\text{by }(\ref{assum1})}\leq 
\int_{\Omega} (\gamma_1+1) \psi_s {\mathbf 1}_{A_{s,l}^+} \label{estIII}
\end{equation}
$\bullet~$ At this point by using the estimates (\ref{estI}), (\ref{estII}) and (\ref{estIII}) together with (\ref{xx}) and (\ref{h-rho}), we obtain: 
\begin{equation} 
\kappa_0 \int_{\Omega} | \nabla \psi_s |^2 \leq \int_{\Omega} 
\underbrace{(\gamma {\mathbf 1}_{A_{s,l}^-}+(\gamma_1+1) {\mathbf 1}_{A_{s,l}^+})}_{=:E} \psi_s. \label{xx1}
\end{equation}  
Note that $E\in L^r(\Omega)$ and:  
$
\|E\|_{L^r} \leq C_0, \quad C_0=\|\gamma\|_{L^r}+\|\gamma_1+1\|_{L^r}.  
$
On the other hand (see Lemma \ref{lemrep}) there exists $\tilde{E} \in (L^{\beta}(\Omega))^3$ satisfying: 
$
\|\tilde{E}\|_{(L^{\beta})^N} \leq \|E\|_{L^r}, \text{ and } \int_{\Omega} E \varphi = \int_{\Omega} \tilde{E} \nabla \varphi 
\quad \forall \varphi \in H_0^1.$
Recall also that we have assumed in (\ref{r}) that $r>\frac{N}{2}$ which implies $\beta > N$. By again using the H\"{o}lder inequality we obtain 
for $\varphi \in H_0^1$: 
$$\int_{\Omega} \tilde{E} \nabla \varphi \leq \|\tilde{E}\|_{(L^\beta)^N} \|\varphi\|_{H_0^1} 
|\text{Supp } \varphi|^{\frac{\beta-2}{2\beta}}.$$
Consequently (\ref{xx1}) leads to: 
\begin{eqnarray*}
\kappa_0 \|\psi_s\|^2_{H_0^1} &\leq& \int_{\Omega} \tilde{E} \nabla \psi_s \leq   C_0 |A_{s,l}|^{\frac{\beta-2}{2\beta}} \|\psi_s\|_{H_0^1} \\ 
&\overset{\text{Young ineq.}}{\leq}& \frac{\kappa_0}{2}\|\psi_s\|^2_{H_0^1} + \frac{C_0^2}{2\kappa_0} |A_{s,l}|^{\frac{\beta-2}{\beta}}.
\end{eqnarray*}
Let $\Phi:=\frac{\beta-2}{\beta} > \frac{N-2}{N}$. We have obtained the estimate: 
$$\|\psi_s\|^2_{H_0^1} \leq \tilde{C_0} |A_{s,l}|^{\Phi}, \quad \tilde{C_0}=\frac{C_0^2}{4 \kappa_0^2}.$$ 
By  now using the Stampacchia estimates (see Subsection \ref{appendixII}) we obtain the existence of 
a real $\Lambda$ independent of $l$ such that $|A_{\Lambda,l}|=0$. Hence: 
$
\|\tilde{v_l} \|_{L^\infty(\Omega)} \leq \Lambda, \quad \Lambda=C(\|\zeta_0\|_{L^{\infty}(\partial \Omega)},|\Omega|,N,r) \frac{C_0}{\kappa_0}$
\end{itemize} 
\end{proof}
{\bf Proof of Proposition \ref{aux1}}
\begin{itemize} 
\item[i)] Existence and estimates. \noindent \\ 
Let $(v_l)_{l\geq 1}$ be the sequence given in Lemma \ref{lem0}. Let also $l\geq 1$ be given and $\tau_l:=\frac{1}{h(e^l)}$. We assume that $\|\gamma_2\|_{L^r} \leq \tau_l$. It follows from Lemma \ref{lem0}.ii) that 
$\|v_l\|_{L^\infty} \leq K$, where $K=K(\kappa_0,DATA)$ (K independent of $l$) is the integer defined by $K=[\Lambda]+1$. \\
Let now $\tau:=\tau_K$ and assume that $\|\gamma_2\|_{L^r} \leq \tau$. Then we have: $\|v_K\|_{L^\infty} \leq K$. Hence $T_K(v_K)=v_K$. On the other hand  
$v_K$ satisfies $(S_K)$, that is: 
$$\quad\left\{
	\begin{array}{l} 
        \rho{\bf u}\nabla e^{v_K} - \div((\nu e^{v_K}+\kappa) \nabla v_K) = g(x,e^{v_K}) \quad \text{in } \Omega \\ 
        v_K = \ln \zeta_0 \quad \text{on } \partial \Omega 
	\end{array}
	\right. 
$$
Let $\zeta:=e^{v_K}$. We have (see Theorem \ref{stam-lip} in the Appendix) $\zeta \in H^1(\Omega) \cap L^{\infty}(\Omega)$ and  
$\nabla \zeta = \zeta \nabla v_K$. Consequently $\zeta$ is a solution of problem (S). \\ 
Moreover $\|\zeta\|_{L^{\infty}(\Omega)} \leq e^K = C(\kappa_0,DATA)$. \\ 
On the other hand $v_K \geq -K$ almost everywhere, implies that $\zeta \geq e^{-K}$ a.e. in $\Omega$. 
Hence we obtain (\ref{Linfty-1}) by setting $\zeta_{\text{min}} = e^{-K}$ and $\zeta_{\text{max}} = e^{K}$. \\ 
$\bullet$ The estimation (\ref{Linfty-2}) is obtained as follows. \\ 
By using the test function $\psi_s^{+}$ instead of $\psi_s$ we obtain: 
\begin{equation} 
\zeta_{\text{max}} \leq \Lambda_1:=e^{C\frac{\|\gamma_1\|_{L^r}}{\kappa_0}}. \label{Linfty-2.0}
\end{equation} 
 This last estimation is only a first step in order to obtain the majoration for $\zeta_{\text{max}}$ announced in (\ref{Linfty-2}). \\   
In fact, let $\tilde{\kappa}:= \nu + \frac{\kappa}{\zeta}$. We have $0 < \nu +  \frac{\kappa_0}{\Lambda_1} =:\tilde{\kappa}_0\leq \tilde{\kappa} \leq  \nu +  \frac{\kappa_0}{\zeta_{\text{min}}} < \infty$, and 
$$
\quad\left\{
	\begin{array}{l} 
	\rho{\bf u} \cdot \nabla \zeta -\div(\tilde{\kappa} \nabla \zeta ) = g^{-}+g^{+} \quad \text{in } \Omega \\ 
     \zeta = \zeta_0 \quad \text{on } \partial \Omega 
	\end{array}
	\right. 
$$
we can then consider the decomposition $\zeta = \zeta_1 + \zeta_2,$ 
where  $\zeta_1$ (resp. $\zeta_2$) $\in H^1(\Omega) \cap L^{\infty}(\Omega)$ satisfies the following problem ($S_1$) (resp. ($S_2$)) : 
\begin{align*} 
&\text{($S_1$)}  \quad\left\{
	\begin{array}{l} 
	\rho {\bf u} \cdot \nabla \zeta_1 -\div (\tilde{\kappa} \nabla \zeta_1 ) = g^{-} \quad \text{in } \Omega \\ 
     \zeta_1 = \zeta_0 \quad \text{on } \partial \Omega 
	\end{array}
	\right.\\  
&\text{($S_2$)}  \quad\left\{
	\begin{array}{l} 
	\rho {\bf u} \cdot \nabla \zeta_2 -\div( \tilde{\kappa} \nabla \zeta_2 ) =  g^{+} \quad \text{in } \Omega \\ 
     \zeta_2 = 0 \quad \text{on } \partial \Omega 
	\end{array}
	\right. 
\end{align*}
\noindent 
Note that the second member in the PDE in ($S_1$) is negative. Hence, by the maximum principle (see \cite{stam} p.80 or \cite{brezis} p.191 for a simplified situation) we obtain: $\zeta_1 \leq \|\zeta_0\|_{L^{\infty}(\partial \Omega)}  \text{ a.e. in }\Omega.$ \\ 
By using next the Stampacchia technique (see again Subsection \ref{appendixII}, Remark \ref{remStam}) we major the function $\zeta_2$ as follows: 
$$
\zeta_2 \leq C' \frac{\|g^+\|_{L^r}}{\tilde{\kappa}_0} \underset{\text{by }(\ref{Linfty-2.0})} \leq  
e^{C\frac{\|\gamma_1\|_{L^r}}{\kappa_0}} \frac{\|g^+\|_{L^r}}{k_0}, \quad C=C(DATA).
$$
This leads to the majoration (\ref{Linfty-2}).
\item[ii)] Regularity results. \noindent \\ 
$\bullet$ If we assume that ${\bf u} \in (L^\infty(\Omega))^{N}$ then $\rho {\bf u} \in (L^\infty(\Omega))^{N}$. Moreover by using the estimates (\ref{Linfty-1}) the diffusion coefficient $\nu + \frac{\kappa}{\zeta}$ is bounded from above and below and 
$g(x,\zeta)\in L^r(\Omega)$ with $r>3/2$ fixed. Hence (see Lemma \ref{lemrep}): $g(x,\zeta)\in W^{-1,\beta}(\Omega)$ with $\beta>N$. Consequently, by using the De Giorgi-Nash Theorem (see \cite{gil} Th. 8.22) we obtain: $\zeta \in \mathcal{C}^{0,\alpha}(\overline{\Omega})$, for some $\alpha > 0$. \\ 
$\bullet$ Assume that in addition we have:
\begin{align*}
& \partial \Omega, \zeta_0 \text{ are of class } \mathcal{C}^{2,\alpha}, \quad g\in\mathcal{C}^{0,\alpha}(\overline{\Omega}\times\mathbb R^{+}), \\ 
& \rho {\bf u} \in (\mathcal{C}^{1,\alpha}(\overline{\Omega}))^N, \quad \kappa \in \mathcal{C}^{1,\alpha}(\overline{\Omega}).
\end{align*} 
We have proved in the previous point that $\zeta \in \mathcal{C}^{0,\alpha}(\overline{\Omega})$. We now iterate the Schauder estimates as follows. In a first step we see that 
$g(x,\zeta)$ and  $\nu + \frac{\kappa}{\zeta}$ are in $\mathcal{C}^{0,\alpha}(\overline{\Omega})$, and by applying 
Theorem 2.7 in \cite{chen} we obtain $\zeta \in \mathcal{C}^{1,\alpha}(\overline{\Omega})$. Consequently we now 
obtain (see Appendix B in \cite{dreyfuss}) $g(x,\zeta), \nu + \frac{\kappa}{\zeta} \in \mathcal{C}^{1,\alpha}(\overline{\Omega})$ and by using Theorem 2.8 in \cite{chen} we finally obtain $\zeta \in \mathcal{C}^{2,\alpha}(\overline{\Omega})$. 
Hence $\zeta$ is a classical solution of (S). 
\end{itemize}
\endproof
\subsection{Proof of Proposition \ref{pro1}} 
Let $n\in\mathbb N$, and $\theta_n,\phi_n$ be given. We assume that condition ($K_n$) is satified. Recall that this implies in particular: $\phi_n \leq \phi_{\text{max}}$, where $\phi_{\text{max}} > \|b\|_ {L^{\infty}(\partial \Omega)}$ was fixed in (\ref{phimax}). Hence, let $\epsilon:=\phi_{\text{max}} -\|b\|_ {L^{\infty}(\partial \Omega)} > 0.$ \\ 
- {\textit Step 1}: We introduce 
$$\kappa^{(n)}:=\frac{C_{\theta}^{(n)}}{\phi_n}, \quad g_{\theta}(x,u):=C_5-C_3F u^2 + C_4 D u.$$
Hence the first subproblem in ($Q_n$) reads as: 
$$
\text{($Q_n.1$)}\quad\left\{
	\begin{array}{l} 
     \rho {\bf u} \cdot \nabla \theta_{n+1} -\div\left( (\nu+\frac{\kappa^{(n)}}{\theta_{n+1}}) \nabla \theta_{n+1} \right) = 
     g_{\theta}(x,\theta_{n+1}) \quad \text{in }\Omega \\  
	 \theta_{n+1} = a \quad \text{on } \partial \Omega
	\end{array}
	\right. 
$$
 Moreover, it is easy to verify that:
\begin{align*}
0 < &\underbrace{\frac{\alpha_{\theta}}{\varphi_{\text{max}}}}_{=:\kappa_0} \leq \kappa^{(n)} \leq 
\| \frac{C_{\theta}^{(n)}}{\phi_n} \|_{L^\infty}< \infty  \\ 
 & g_{\theta}^{+}(x,u)=C_5 + C_4 D^+ u, \quad  g_{\theta}^{-}(x,u)=-C_3Fu^2 + C_4 D^- u.
\end{align*}
Note also that $\kappa_0$ is independent of $n$. 
Hence we can apply Proposition \ref{aux1}.i) (take $\zeta=\theta_{n+1},\kappa=\kappa^{(n)},g=g_{\theta},\gamma_1=C_5,\gamma_2=C_4D^{+},h(t)=t,
\gamma=C_3F+C_4|D^{-}|,\zeta_0=a$). We obtain the existence of $\tau_0>0$ independent of $n$ such that 
if $\|D^{+}\|_{L^r} \leq \tau_0$ then problem ($Q_n.1$) admits at least one weak solution $\theta_{n+1} \in H^1(\Omega) \cap L^{\infty}(\Omega)$. Moreover we have 
\begin{equation} 0 < e^{-C\phi_{\text{max}}} \leq \theta_{n+1} \leq  e^{C\phi_{\text{max}}} < \infty, \quad C=C(DATA). \label{est0}
\end{equation} 
- {\textit Step 2}: Let now 
$$\tilde{\kappa}^{(n)}:=\frac{C_{\phi}^{(n)}}{\theta_{n+1}}, \quad g_{\phi}^{(n)}(x,u)=g_{\phi}^{(n)}(x)=
 -\phi_n(C_8 \theta_{n+1}^{-1}+ C_6 \theta_{n+1} F- C_7 D).$$
With these notations, the second subproblem in ($Q_n$) reads as: 
$$
\text{($Q_n.2$)}\quad\left\{
	\begin{array}{l} 
     \rho {\bf u} \cdot \nabla \phi_{n+1} -\div\left( (\nu+\frac{\tilde{\kappa}^{(n)}}{\phi_{n+1}}) \nabla \phi_{n+1} \right) = 
     g_{\phi}^{(n)} \quad \text{in }\Omega \\  
	 \phi_{n+1} = b \quad \text{on } \partial \Omega
	\end{array}
	\right. 
$$
 We verify that:
\begin{align*}
0 < & \underbrace{\frac{\alpha_{\phi}}{\theta_{\text{max}}}}_{=:\tilde{\kappa}_0} \leq \tilde{\kappa}^{(n)} \leq \frac{\| C_{\phi}^{(n)} \|_{L^\infty}}{\theta_{\text{min}}} < \infty, \quad g_{\phi}^{+}=\phi_n C_7 D^+ \leq \phi_{\text{max}} C_7 D^+, \\ 
& g_{\phi}^{-}=-\phi_n(C_8 \theta_{n+1}^{-1}+ C_6 \theta_{n+1} F-C_7 D^-), \quad |g_{\phi}^{-}| \overset{by (\ref{est0})}{\leq} \tilde{C}(DATA,\phi_{\text{max}}).
\end{align*}
Hence we can apply the proposition \ref{aux1}.i) (take now $\zeta=\phi_{n+1},\kappa=\tilde{\kappa}^{(n)},g=g_{\phi}^{(n)},\gamma_1=\phi_{\text{max}}C_7D^{+},h=0,
\gamma=\tilde{C},\zeta_0=b$). Then we obtain the existence of a 
weak solution\footnote{at this stage there is anymore additional condition needed for $\|D^{+}\|_{L^r}$ because $h=0$} $\phi_{n+1} \in H^1(\Omega) \cap L^{\infty}(\Omega)$ for problem \text{($Q_n.2$)}. Moreover we have 
\begin{equation} 0 < e^{-C'e^{C \phi_{\text{max}}}} \leq \phi_{n+1} \leq  e^{C'e^{C \phi_{\text{max}}}} < \infty, \label{est0}
\end{equation}
where C,C' depend on DATA but not on $n$.
Moreover, by using (\ref{Linfty-2}) we have: 
\begin{align}
\|\phi_{n+1}\|_{L^{\infty}(\Omega)} &\leq \|b\|_{L^{\infty}(\partial \Omega)} + C''\|g_{\phi}^{+}\|_{L^r} e^{C \phi_{\text{max}}} 
\|\phi_{n+1}\|_{L^{\infty}(\Omega)}, \quad C''=C''(DATA), \notag \\ 
&\leq \|b\|_{L^{\infty}(\partial \Omega)} + \underbrace{C'' \|C_7\|_{L^\infty} e^{C \phi_{\text{max}}}
}_{:=K(\phi_{\text{max}},DATA)} \|D^{+}\|_{L^r} \|\phi_{n+1}\|_{L^{\infty}(\Omega)}. \label{est-max}
\end{align}
Assume now that 
$$ \|D^{+}\|_{L^r} \leq \tau:=\min(\tau_0,\frac{\epsilon}{\phi_{\text{max}} K}).$$
Then (\ref{est-max}) leads to 
$$\|\phi_{n+1}\|_{L^{\infty}(\Omega)} \leq \|b\|_{L^{\infty}(\partial \Omega)} + \frac{\epsilon}{\phi_{\text{max}}}  \|\phi_{n+1}\|_{L^{\infty}(\Omega)},$$
and it follows: 
$$\|\phi_{n+1}\|_{L^{\infty}(\Omega)} \leq \frac{\|b\|_{L^{\infty}(\partial \Omega)}}{1-\frac{\epsilon}{\phi_{\text{max}}}} \leq 
\phi_{\text{max}}.$$
- {\textit Final Step}: if we assume that $\|D^{+}\|_{L^r}\leq\tau$ then by using the results established in the previous two steps, we conclude that there exists a solution $(\theta_{n+1},\varphi_{n+1}) \in (H^1\cap L^{\infty})^2$ for problem ($Q_n$). Moreover 
this solution satisfies ($K_{n+1}$) and (\ref{est1_prop}),(\ref{est2_prop}) hold. \endproof
\section{Proofs of the theorems}
We begin by a lemma: 
\begin{lemm} \label{lem22}
Under the assumptions of Proposition \ref{pro1}, we can extract a subsequence 
(still denoted by $(\theta_n,\phi_n)$) such that 
\begin{align}
&\theta_n \overset{*}{\rightharpoonup}\theta, \phi_n\overset{*}{\rightharpoonup}\phi \ \text{ in } L^{\infty}(\Omega), \qquad \theta_n \rightharpoonup\theta, \phi_n \rightharpoonup \phi  \text{ in } H^1(\Omega), \quad \label{lim1} \\  
&\frac{1}{\theta_n \phi_n}\to\frac{1}{\theta \phi}, C_{\theta}^{(n)}\to C_{\theta}(x,\theta,\phi), 
C_{\phi}^{(n)}\to C_{\phi}(x,\theta,\phi) \text{ in } L^{p}(\Omega), \ p<\infty \label{lim2} 
\end{align}
\end{lemm}
\begin{proof}
The first properties in (\ref{lim1}) follow directly from Proposition \ref{pro1}. By next using  
$\theta_n-\theta_0$ as test function in ($Q_n$.1) and $\phi_n-\phi_0$ as test function in ($Q_n$.2) we obtain a 
uniform bound for $(\theta_n)$ and  $(\phi_n)$ in the $H^1$-norm. Hence the second properties in (\ref{lim1}) follow. Finally 
property (\ref{lim2}) is obtained by using the dominated convergence theorem. In fact we have: 
\begin{align*} 
&\frac{1}{\theta_n \phi_n} \to \frac{1}{\theta \phi},\  C_{\theta}^{(n)} \to  C_{\theta}(x,\theta,\phi), \ C_{\phi}^{(n)}\to 
 C_{\phi}(x,\theta,\phi) \quad \text{a.e in }\Omega, \\ 
&|\frac{1}{\theta_n \phi_n}| \leq \frac{1}{\theta_{\text{min}} \phi_{\text{min}}}, \  
|C_{\theta}^{(n)}|\leq\sup_{(v,w)\in\mathcal{K}}C_{\theta}(x,v,w),\  |C_{\phi}^{(n)}|\leq\sup_{(v,w)\in\mathcal{K}} C_{\phi}(x,v,w),
\end{align*}
where $\mathcal{K}=[0,\theta_{\text{max}}]\times[0,\phi_{\text{max}}]$.
\end{proof}
\subsection{Proof of Theorem \ref{theo1}} 
By using (\ref{lim1}) together with (\ref{lim2}) we obtain: 
\begin{eqnarray*}
\rho {\bf u}.\nabla \theta_{n+1} &\rightharpoonup& \rho {\bf u}.\nabla \theta \ \text{ in } L^1(\Omega) \\
(\nu+\frac{C_{\theta}^{(n)}}{\theta_{n+1} \varphi_n}) \nabla \theta_{n+1} &\rightharpoonup& (\nu+\frac{C_{\theta}}{\theta \varphi}) \nabla \theta \ \text{ in } L^q(\Omega), \ q<2 \\
\rho {\bf u}.\nabla \phi_{n+1} &\rightharpoonup& \rho {\bf u}.\nabla \phi \ \text{ in } L^1(\Omega) \\ 
(\nu+\frac{C_{\phi}^{(n)}}{\theta_{n+1} \varphi_{n+1}}) \nabla \phi_{n+1} &\rightharpoonup& (\nu+\frac{C_{\phi}}{\theta \varphi}) \nabla \phi \ \text{ in } L^q(\Omega),\ q<2.
\end{eqnarray*}
Moreover by using (\ref{lim1}) together with the property $\theta_{n+1}\geq \theta_{\text{min}}>0$ we obtain: 
$$\theta_{n+1}^2 \to \theta^2, \ \theta_{n+1} \to \theta, \ \theta_{n+1}^{-1} \to \theta^{-1} \quad \text{in } L^p(\Omega), \ p<\infty.$$    
Hence we can pass to the limit in the approximate problems ($Q_n$). 
We obtain a weak solution $(\theta,\phi)$ for problem (Q). That is for all $\psi\in\mathcal{D}(\Omega)$:   
\begin{align*} 
	\int_{\Omega} \rho {\bf u} \cdot \nabla \theta \ \psi + \int_{\Omega} (\nu+\frac{C_{\theta}}{\theta \phi}) \nabla \theta \cdot \nabla \psi &= 
	\int_{\Omega} (C_5 - C_3  F \theta^2+ C_4 \theta D) \psi \\ 
  \int_{\Omega} \rho {\bf u} \cdot \nabla \phi \ \psi + \int_{\Omega} (\nu+\frac{C_{\phi}}{\theta\phi}) \nabla \phi \cdot \nabla \psi &= 
     \int_{\Omega} -\phi (C_8 \theta^{-1}+ C_6 \theta F - C_7 D)\psi \\  
	 \theta = a, \phi = b \quad \text{on } \partial \Omega 
\end{align*}
Moreover this solution satisfies:    
$$
\theta, \phi \in H^1(\Omega) \cap L^{\infty}(\Omega),\quad \theta, \phi \geq \min(\theta_{\text{min}},\phi_{\text{min}}) > 0 \ \text{ a.e. in }  \Omega.
$$
\subsection{Proof of Theorem \ref{theo2}} \label{theo2D}

When $N=2$ the function $F$ has a stronger property of positivity (see lemma \ref{lemF} in the Appendix):  
\begin{equation} 
F\geq \frac{D^2}{3}. \label{MF}
\end{equation}
We will see that this last property allows one to obtain a weak solution for problem (Q) under the assumptions ($A_0$) and ($A_1$) but without assuming a low compressibility condition of the form (\ref{lowcomp}).\\
In order to prove this result, we take over the proof of Proposition \ref{pro1} with slight modifications: 
if $(\theta_n,\phi_n)$ is given and satisfies $(H_n)$ (it is not useful to consider ($K_n$) here) then problem ($Q_n$) has at least one solution 
$(\theta_{n+1},\phi_{n+1})$ satisfying in addition $(H_{n+1})$ and the estimates (\ref{est1_prop}),(\ref{est2_prop}). \\ 
- {\textit Step 1}: By using property (\ref{MF}) we majore $g_{\theta}^{+}$ as follows: 
\begin{align*} 
g_{\theta}^{+}(x,u) &= C_5 + C_4 D^{+}u - C_3 F u^2 \leq  C_5 + C_4 D^{+}u - \frac{C_3}{3} (D^{+} u)^2  \\ 
&\leq C_5 +\underbrace{ D^{+} u}_{\geq 0}\underbrace{(C_4- \frac{C_3}{3} (D^{+} u)}_{\leq 0 \text{ if } D^{+} u \geq \frac{3C_4}{C_3}} 
\leq C_5+\frac{3C_4^2}{C_3}.
\end{align*}
Hence we have here estimated $g_{\theta}^{+}$ independtly of the second variable. We then apply Proposition \ref{aux1}.i), but now we take 
$\gamma_1=C_5+\frac{3C_4^2}{C_3}$ instead of $C_5$ and $h(t)=0$ instead of $h(t)=t$. It follow that there exists 
(without any condition on  $\|D^{+}\|_{L^r}$ because $h=0$) a weak solution $\theta_{n+1} \in H^1(\Omega) \cap L^{\infty}(\Omega)$  for problem ($Q_n.1$), 
with the estimate: 
\begin{equation} 0 < e^{-C\|\phi_n\|_{L^\infty}} \leq \theta_{n+1} \leq  e^{C\|\phi_n\|_{L^\infty}} < \infty, \quad C=C(DATA).\label{est0-2D}
\end{equation}
- {\textit Step 2}: By taking over the arguments presented in the proof of Proposition \ref{pro1} we see that 
problem ($Q_n.2$) has at least one positive solution $\phi_{n+1} \in H^1(\Omega) \cap L^{\infty}(\Omega)$. \\  
Hence at this point we have obtained a weak solution $(\theta_{n+1},\phi_{n+1})$ for ($Q_n$) satisfying in addition $(H_{n+1})$. It remains 
to prove that the estimates  (\ref{est1_prop}),(\ref{est2_prop}) hold. We have made a first step in this direction by proving (\ref{est0-2D}). 
We will now prove: 
\begin{equation} 0 < e^{-C\|\theta_{n+1}\|_{L^\infty}} \leq \phi_{n+1} \leq  \|b\|_{L^{\infty}(\partial \Omega)}, \quad C=C(DATA).\label{est1-2D}
\end{equation}
In fact, by using the additional assumption ($A_1$) we majore the function $g_{\phi}^{(n)}$ as follows: 
\begin{align*} g_{\phi}^{(n)}(x)&=
 -\phi_n(C_8 \theta_{n+1}^{-1}+ C_6 \theta_{n+1} F- C_7 D) \overset{\text{by }(\ref{MF})}{\leq} 
-\phi_n(C_8 \theta_{n+1}^{-1}+ \frac{C_6}{3} \theta_{n+1} D^2- C_7 D) \\ 
&\leq -\frac{\rho\phi_n}{\theta_{n+1}}(c_8+\frac{c_6}{3}(\theta_{n+1} D)^2-c_7 (\theta_{n+1} D)) = 
-\frac{\rho\phi_n}{\theta_{n+1}}\mathcal P(\theta_{n+1} D),
\end{align*}
with $\mathcal P(X):=\frac{c_6}{3}X^2-c_7 X+c_8$. We remark that the discriminant $\Delta$ of $\mathcal P$ is negative: 
$\Delta=c_7^2-\frac{4}{3}c_6c_8=-4.864*10^{-3}<0$. It follows that $\mathcal P$ is positive and consequently $g_{\phi}^{(n)}$ is negative. 
Hence by applying Proposition \ref{aux1}.i) with now $g^+\equiv 0$ we obtain (\ref{est1-2D}). \\
- {\textit Final Step}:  By using (\ref{est0-2D}) together with (\ref{est1-2D}) we obtain the estimates (\ref{est1_prop}) and (\ref{est2_prop}). Hence we have 
recovered the conclusions of Proposition \ref{pro1}. The remaindee of the proof for Theorem \ref{theo3} 
is exactly the same as for Theorem \ref{theo1}: we can extract a subsequence with the properties (\ref{lim1})-(\ref{lim1}). These properties are 
sufficient to pass to the limit $n\to\infty$ in ($Q_n$) and Theorem \ref{theo2} follows.
\subsection{Proof of Theorem \ref{theo3}} 
Let ($\theta,\phi$) be a weak solution  for (Q) in the class $\mathcal S$ and consider the notations: 
$$g_{\theta}(x,u):=C_5-C_3F u^2 + C_4 D u, \quad  g_{\phi}(x,u):=-u(C_8 \theta^{-1}+ C_6 \theta F- C_7 D).$$
\begin{itemize}
\item[i)] It suffices to remark that the coefficients $\frac{C_{\theta}}{\phi}$ and $\frac{C_{\phi}}{\theta}$ are bounded from above and below, and $g_{\theta}, g_{\phi}$ are Caratheodory functions satisfying (\ref{G2}) and (\ref{G3}). Hence we can apply the first point in Proposition \ref{aux1}-ii) in each equation of (Q). We obtain: $\theta, \phi \in \mathcal{C}^{0,\alpha}(\overline{\Omega})$, for some $\alpha > 0$. 
\item[ii)]Assume that in addition we have: 
\begin{align*}
& \partial \Omega,a,b \text{ are of class } \mathcal{C}^{2,\alpha}, \quad F \in \mathcal{C}^{0,\alpha}(\overline{\Omega}), \\ 
& \rho {\bf u} \in (\mathcal{C}^{1,\alpha}(\overline{\Omega}))^N, \quad C_{\theta}, C_{\phi} \in \mathcal{C}^{1,\alpha}(\overline{\Omega}\times(\mathbb R^{+})^2).
\end{align*} 
We remark now that the conditions in the second part of Proposition \ref{aux1}-ii) are satisfied for each equation of (Q). Hence $\theta, \phi \in \mathcal{C}^{2,\alpha}(\overline{\Omega})$ and it is a classical solution of (Q). 
\end{itemize}
\section{Appendix}
\subsection{Derivation of the $\theta-\phi$ model.}
The model is constructed from the $k-\epsilon$ one which takes the form:    
\begin{align}
&\partial_t k + {\bf u} \cdot \nabla k - \frac{c_{\nu}}{\rho} \text{div }((\nu +\rho\frac{k^2}{\varepsilon})\nabla k)=
c_\nu \frac{k^2}{\epsilon}F-\frac{2}{3}kD-\epsilon, \label{a-k1}\\ 
&\partial_t \epsilon + {\bf u} \cdot \nabla \epsilon - \frac{c_{\epsilon}}{\rho} \text{div }((\nu+\rho\frac{k^2}{\varepsilon})\nabla \varepsilon)=
c_1kF-\frac{2c_1}{3c_\nu}\epsilon D-c_2\frac{\epsilon^2}{k}, \label{a-eps1}
\end{align}
where $D(x,t):=\div {\bf u}(x,t), \ F(x,t):=\frac{1}{2}|\nabla {\bf u}+ (\nabla {\bf u})^T|^2-\frac{2}{3}D(x,t)^2\geq 0$ (see the next subsection )  and $c_\nu, c_\epsilon, c_1, c_2$ are generally taken as positive constants. Their usual values 
are (see \cite{moha} p.122): 
\begin{equation} c_{\nu}=0.09,  \ c_{\epsilon}=0.07, \ c_1=0.128, \ c_2=1.92. \label{val1} 
\end{equation}
We then consider the new variables
\begin{equation}\theta=\frac{k}{\epsilon}, \quad \phi=k^{\alpha}\epsilon^{\beta},\label{TF}
\end{equation}
with $\alpha$ and $\beta$ to be chosen appriopriately. 
Let $D_t$ denote the total derivative operator. By using (\ref {a-k1}) together with 
(\ref{a-eps1}) we obtain an equation for $\theta$: 
\begin{align} 
D_t \theta &= \frac{\partial \theta}{\partial t} + {\bf u}\cdot \nabla \theta = \frac{1}{\epsilon} D_t  k - \frac{k}{\epsilon^2}  
D_t \epsilon = -c_3 \theta^2  F + c_4 \theta D + c_5 + \text{Diff}_{\theta}, \notag \\ 
& c_3=c_1-c_{\nu}, \ c_4=\frac{2}{3}(\frac{c_1}{c_{\nu}}-1), \ c_5=c_2-1, \label{val2}
\end{align}
where $\text{Diff}_{\theta}$ denotes the collected terms coming from the viscous one in the $k$ and $\epsilon$ equations.
The equation for $\phi$ is obtained in the same way:
\begin{align} 
D_t &\phi = \alpha k^{\alpha-1}\epsilon^{\beta} D_t  k + \beta k^{\alpha}\epsilon^{\beta-1} D_t \epsilon =   
\alpha k^{\alpha-1}\epsilon^{\beta}(c_{\nu} \frac{k^2}{\epsilon}F-\frac{2}{3}kD-\epsilon) +  \text{Diff1}_{\phi} \notag \\
&+ \beta k^{\alpha}\epsilon^{\beta-1}(c_1 kF-\frac{2c_1}{3c_{\nu}}\epsilon D-c_2 \frac{\epsilon^2}{k}) +  \text{Diff2}_{\phi} = F k^{\alpha+1}\epsilon^{\beta-1}(\alpha c_{\nu}+\beta c_1) \notag \\ 
&- k^{\alpha}\epsilon^{\beta} D\frac{2}{3}(\alpha+\beta\frac{c_1 }{c_{\nu}})-k^{\alpha-1}\epsilon^{\beta+1}(\alpha+\beta c_2 )+\text{Diff}_{\phi}, \notag \\ 
&=-\phi(c_6 \theta F-c_7 D + c_8 \theta^{-1}) + \text{Diff}_{\phi}, \notag \\ 
&c_6=-\alpha c_{\nu}-\beta c_1, \ c_7=-\frac{2}{3}(\alpha+\beta\frac{c_1}{c_{\nu}}), \ c_8=\alpha+\beta c_2, \label{val3}
\end{align}
where $\text{Diff}_{\phi}=\text{Diff1}_{\phi}+\text{Diff2}_{\phi}$ is the sum of the terms coming from the viscous one in the $k$ and $\epsilon$ equations. 

The usual constant values for the parameters $c_3,c_4,c_5$ appearing in the model are obtained by replacing 
the values (\ref{val1}) in the expressions (\ref{val2}). This leads to: 
$$ c_3=0.038, c_4=0.2815, c_5=0.92.$$

At this stage it remains to choose appriopriately $\alpha$ and $\beta$ in (\ref{TF}), and to model the terms  
$\text{Diff}_{\theta}$ and $\text{Diff}_{\phi}$. \\ 
It is schown in \cite{moha}, page 67, that a good choice in the incompressible situation (i.e. when $D=0$) is for instance 
$\alpha=-3, \beta=2$. This leads to the following constant values:
\begin{equation} 
c_6=0.014, \ c_7=0.104, \ c_8=0.84, \label{val4}
\end{equation}
and this makes the dynamic stable for the equation in $\phi$. I.e. in the absence of the viscous part $\text{Diff}_{\phi}$, we have: 
$D_t \phi \leq 0$. \\ 
In the compressible situation the authors suggest in \cite{moha}, page 125 to consider another choice: 
$\alpha=-2\frac{c_1}{c_{\nu}}\approx -2.88$ and $\beta=2$ which makes again the dynamic stable. \\ 
Nevertheless in this last situation, the variable $\phi$ does not have a clear physical meaning (whereas when $\alpha=-3, \beta=2$ 
we have $\phi=\frac{\epsilon^2}{k^3}$ and $L:=\phi^{-1/2}$ represents a length scale of turbulence (see \cite{wilcox})). 
Moreover, a carefully estimation schows (see Lemma \ref{lemF} in the next subsection) that when $N=2$ we have 
$F\geq \frac{1}{3}D^2$. This leads to: 
$$ 
-\phi(c_6 \theta F-c_7 D + c_8 \theta^{-1}) \leq -\frac{\phi}{\theta}(\frac{c_6}{3} (\theta D)^2-c_7 (\theta D) + c_8) 
= -\frac{\phi}{\theta}\mathcal{P}(\theta D),$$ 
with $\mathcal{P}(X):=\frac{c_6}{3} X^2-c_7 X+ c_8$. Hence the choice $\alpha=-3, \beta=2$ makes again the dynamic 
stable when $N=2$. In fact, in this case $c_6,c_7,c_8$ take the values (\ref{val4}) and the discriminant $\Delta$ of $\mathcal{P}$ 
is $\Delta=c_7^2-\frac{4}{3}c_6c_8=-4.864*10^{-3} < 0$. Consequently $\mathcal{P}(X) \geq 0$ and 
$D_t \phi \leq 0$ in the absence of the viscous terms. \\ 

In consequence we point out that the choice $\alpha=-3, \beta=2$ is also interesting in the compressible situation. We shall make this choice in all 
the situations. Our analysis (see Theorem \ref{theo1}) shows that this leads to a well posed model even when $N=3$ under an additional assumption of low compressibility of the flow. \\ 

The terms $\text{Diff}_{\theta}$ and $\text{Diff}_{\phi}$ are modelled (see \cite{Lew}) by: 
\begin{equation}
\text{Diff}_{\theta}=\frac{1}{\rho}\div ((\nu + c_{\theta} \nu_t)\nabla \theta), \quad 
\text{Diff}_{\phi}=\frac{1}{\rho}\div ((\nu + c_{\phi} \nu_t)\nabla \theta), \label{diff}
\end{equation}
where $\nu_t:=\rho c_{\nu}\frac{k^2}{\epsilon}=\rho \frac{c_{\nu}}{\theta \phi}$ is the turbulent viscosity coming from 
the equation of $k$, and  $c_{\theta},c_{\phi}$ are two new parameters for the model. \\ 
The determination of the parameters $c_{\theta}$ and $c_{\phi}$ can be realized in the same way as for the determination 
of the coefficients arising in the $k-\epsilon$ model (see \cite{moha}). In \cite{moha-num} a constant value for both $c_{\theta}$ and $c_{\phi}$ was numerically tuned from a simulation of a Poiseuil flow. However better results are obtained if we allow $c_{\theta},c_{\phi}$ to be some positive functions (see \cite{moha2}). In our analysis we allow the coefficients to be of a 
very general form, in particular $c_{\theta},c_{\phi}$ may depend on $x,\theta$ and $\phi$. We only assume that they 
are Caratheodory functions and that they satisfy some positivity and boundedness properties (see (\ref{h-Ci2})-(\ref{h-Ci4}), 
where $C_{\theta}=\rho c_{\theta}$ and $C_{\phi}=\rho c_{\phi}$). 
\subsection{Positivity of the function $F$} \label{appendixIII}

In this paragraph we will establish some properties of positivity for the function $F$ appearing in the models.  \\ 
Let $\mathcal M_N(\mathbb R)$ denote the vector space of the N-square matrix with real coefficients, equipped with the scalar product: 
$$A:B = \sum_i \sum_j a_{ij}b_{ij}, \quad \forall A=(a_{ij}),B=(b_{ij}) \in \mathcal M_N(\mathbb R).$$
Hence $|A|^2:=\sqrt{A:A}$ defines a norm on $\mathcal M_N(\mathbb R)$.

For a vector field ${\bf u}: \Omega \to \mathbb R^N$ we classically defines the gradient $\nabla {\bf u}$  and the divergence $\div {\bf u}$ (=:D) by:
\begin{align*} 
\nabla {\bf u}&: \   \Omega \to \mathcal M_N(\mathbb R), \quad (\nabla {\bf u}(x))_{ij} = \frac{\partial u_i(x)}{\partial x_j}, \\
D&: \   \Omega \to \mathbb R, \quad D(x)=\sum_i \frac{\partial u_i(x)}{\partial x_i} = \text{Tr}(\nabla {\bf u}).
\end{align*}
Recall that the function $F$ was defined by the formula: 
\begin{equation}
F(x):=\frac{1}{2}|\nabla {\bf u}+ (\nabla {\bf u})^T|^2-\frac{2}{3}D(x)^2, \label{defF}
\end{equation}
and by an easy calculation we obtain: 
\begin{equation}
F(x)=(\nabla {\bf u} + (\nabla {\bf u})^t):\nabla {\bf u} -\frac{2}{3}D(x)^2. \label{defF2}
\end{equation}
This last expression is sometimes chosen (for instance in \cite{moha}) to equivalently define $F$. \\ 
The important fact is that we always have $F \geq 0$ but moreover, when $N=2$ the stronger estimate:  $F \geq \frac{1}{3} D^2$ holds. These properties are 
established in the following lemma:

\begin{lemm} \label{lemF}
The function  $F$ satisfies the estimates: 
\begin{align*}
F &= \frac{2}{3}\left( (\partial_{1} u_1-\partial_{2} u_2)^2+(\partial_{1} u_1-\partial_{3} u_3)^2+ (\partial_{2}u_2-\partial_{3}u_3)^2 \right) \\ 
&\ +(\partial_{2}u_1+\partial_{1} u_2)^2 + (\partial_{3}u_1+\partial_{1} u_3)^2 
+(\partial_{3}u_2+\partial_{2}u_3)^2 \geq 0, \quad \text{when } N=3. \\ 
F &= (\partial_{1} u_1-\partial_{2} u_2)^2+(\partial_{2}u_1+\partial_{1} u_2)^2+\frac{1}{3}D^2 \geq \frac{1}{3} D^2, \quad \text{when } N=2.
\end{align*} 
 \end{lemm}
\begin{proof}
Let $N=2$ or 3, and $M:=\nabla {\bf u} + (\nabla {\bf u})^t$. Then $M_{ij} = \partial_{j} u_i +  \partial_{i} u_j$ and we obtain:  
\begin{align*}
|M|^2 &= \sum_i \big( (2  \partial_{i} u_i)^2 + \sum_{j\neq i} (\partial_{j} u_i +  \partial_{i} u_j)^2\big) 
= 4 \sum_i (\partial_{i} u_i)^2 +2\sum_i \sum_{j>i} (\partial_{j} u_i +  \partial_{i} u_j)^2, \\ 
F&=\frac{1}{2}|M|^2-\frac{2}{3}D^2= 
\underbrace{2\sum_i (\partial_{i} u_i)^2 - \frac{2}{3}(\sum_i \partial_{i} u_i)^2}_{:=A} + \sum_i \sum_{j>i} (\partial_{j} u_i +  \partial_{i} u_j)^2 
\end{align*}
The term $A$ is evaluated separately in the cases $N=2$ and $N=3$.\\ 
We remark that $2(a^2+b^2+c^2)-\frac{2}{3}(a+b+c)^2=\frac{2}{3}((a-b)^2+(a-c)^2+(b-c)^2)$. Hence, when $N=3$ we have 
$$A=\frac{2}{3}\left( (\partial_{1} u_1-\partial_{2} u_2)^2+(\partial_{1} u_1-\partial_{3} u_3)^2+ (\partial_{2}u_2-\partial_{3}u_3)^2 \right),$$
and we obtain the expression announced for $F$. \\ 
In the same way, we remark that 
\begin{align*}
2(&a^2+b^2)-\frac{2}{3}(a+b)^2 = \frac{2}{3}(a-b)^2 + \frac{2}{3}(a^2+b^2) \\ 
&=\frac{2}{3}(a-b)^2 + \frac{1}{3}(a+b)^2 +  \frac{1}{3}(a-b)^2 = (a-b)^2 + \frac{1}{3}(a+b)^2.
\end{align*}
Hence, when $N=2$ we obtain: 
$$
A=(\partial_{1} u_1-\partial_{2} u_2)^2 + \frac{1}{3}\underbrace{(\partial_{1} u_1+\partial_{2} u_2)^2}_{=D^2} , \\ 
$$
and the expression for $F$ follows.
\end{proof}
\subsection{On the low compressibility assumption}

 We will show here that without any assumption of low compressibility of the 
form (\ref{lowcomp}), problem (Q) may be very hard to analyze when N=3, and singular solutions or non existence of weak solution may occur. \\ 

When the dimension equals two we have seen in Theorem \ref{theo2} that a condition of low compressibility is not necessary. The reason is related to the fact that a stronger 
property of positivity for $F$ holds in this case, i.e. we have: $F \geq \frac{D^2}{3}$. When the dimension equals 3 such a property does not hold in general. \\ 

In fact, let for instance $\Omega=B_{\mathbb R^3}(0,1)$ and 
$$\rho=(\prod_{i=1}^3 (x_i+4))^{-1}, \quad {\mathbf u}=(u_1,u_2,u_3)^t , \ u_i=x_i+4.$$
Then a simple calculation gives:  $D=3$ and $F=0$. Hence (Q) reads as: 
%\begin{breakbox}
 $$  \quad\left\{
	\begin{array}{l} 
	\rho{\bf u} \cdot \nabla \theta -\div\left( (\nu+\frac{C_{\theta}}{\theta \phi}) \nabla \theta\right) = 3 C_4 \theta+ C_5 \quad \text{in } \Omega \\ 
     \rho {\bf u} \cdot \nabla\phi -\div\left( (\nu+\frac{C_{\phi}}{\theta\phi}) \nabla \phi\right) = 
      \varphi (3 C_7 - C_8 \theta^{-1}) \quad \text{in }\Omega \\ 
	 \theta = a, \phi = b \quad \text{on } \partial \Omega 
	\end{array}
	\right. 
$$
In this situation the problem becomes hard to analyze. Assume however that we have obtained a solution $(\theta,\phi)$ in the class $\mathcal S$. Then the equation 
satisfied by $\theta$ is closely related to: 
$$(R) \quad\left\{
	\begin{array}{l} 
	-\div\left( \eta \nabla \theta\right) = 3 C_4 e^\theta+ C_5 \quad \text{in } \Omega, \\
              \theta = a \quad \text{on } \partial \Omega, 
	\end{array}
	\right. 
$$
with $\eta$ bounded from above and below. Hence a contradiction can occur because the problem (R) may not have any weak solution (see for instance \cite{brezis2,lupo}). 
Note that in the considered  example $\rho$ and $\mathbf{u}$ satsify all the conditions needed in ($A_0$), except 
$\mathbf{u}\cdot \mathbf{n}=0$ on $\partial \Omega$ but this is not restrictive for the purpose here. In fact we can consider the domain $\Omega_1=B_{\mathbb R^3}(0,2)$ which contains $\Omega$ and we can extend $\rho,\mathbf{u}$ in 
$\Omega_1$ in such a way that all the conditions in ($A_0$) are satisfied. Hence we obtain an example within the main situation 
of the study but the evocated problems remain the same.  
\subsection{A generalized chain rule} \label{appendix0}

Let $u \in W^{1,p}(\Omega)$ and $G: \mathbb R \to \mathbb R$ be a Lipschitz function. 
We recall here some useful properties of the composed function $G(u)$. In particular we schall see 
that $G(u) \in W^{1,p}(\Omega)$. Moreover in some situations we also have $\frac{\partial}{\partial x_i} G(u) = G'(u)\frac{\partial u}{\partial x_i}$. \\ 

The main result we have in mind is Theorem \ref{stam-lip} which is due to Stampacchia. In particular we point out that the additional assumption $G(0)=0$ for the Lipschitz function $G$ is only necessary if $\Omega$ is unbounded and $p\neq \infty$, or if we want that $G(u)$ has a vanishing trace on $\partial \Omega$ when $u$ has it (this last situation was in fact the case of interest of Stampacchia). 
\begin{theo} \label{stam-lip}
Let $G$ be a Lipschitz real function, $\Omega \subset \mathbb R^N$ be a bounded open Lipschitz domain, and $u\in W^{1,p}(\Omega)$, with  $p\in[1;\infty]$. We have:
\begin{itemize}
\item[i)] 
$G(u)\in W^{1,p}(\Omega)$. Moreover if $u\in W^{1,p}_0(\Omega)$ and $G(0)=0$ then $G(u)\in W^{1,p}_0(\Omega)$.
\item[ii)]
If $G'$ has a finite number of  discontinuity\footnote{the derivative $G'$ of $G$ takes here the classical sense}. Then the 
weak derivatives of $G(u)$ are given by the formula:
\begin{equation}
\frac{\partial}{\partial x_i} G(u)=G'(u) \frac{\partial u}{\partial x_i} \quad \text{a.e. in } \Omega.  \label{der-stam}
\end{equation}
\end{itemize}
\end{theo}
\begin{proof}
See the appendix in \cite{stam} for the original proof or \cite{gil} Theorem 7.8 and \cite{evans} Theorem 4 for alternative 
proofs and additional comments.  \\ 
We also recall that the formula (\ref{der-stam}) may be interpreted in 
some critical points. In fact, let $(t_i)_{i=1,..,n}$ denote the points of discontinuity of $G'$ and let 
$E_i:=\{x \in \Omega: u(x)=t_i\}$ be the associated level sets for the function $u$. Let $1\leq i\leq n$ be a fixed integer. If $|E_i|>0$ 
then the formula (\ref{der-stam}) has a priori no sense in this last set which is not negligible. Nevertheless it can be shown (see \cite{stam}) 
that $\frac{\partial u}{\partial x_i}=0$ on such a set. Hence we interpret the right hand side of (\ref{der-stam}) as zero in the 
critical set $E_i$.   
\end{proof} 

We now establish some technical results used in the proof of Lemma \ref{lem0}. 
Let $v_l\in H^1(\Omega)$ be given and consider the function $h_{\pm}\in L^{\mathbf 1}_{loc}(\mathbb R)$ defined by: 
\begin{equation} 
h_{\pm}(y):=e^y (y-T_{s,\pm}(y)). \label{h}
\end{equation} 
We introduce the functions  
\begin{eqnarray*}
g_{\pm}(t) &:=& \int_0^{T_l(t)} h_{\pm}(y)dy,\quad t \in \mathbb R \\ 
\zeta_{l,\pm}(x) &:=& g_{\pm}(v_l(x)) \quad \text{a.e. }x \in \Omega.
\end{eqnarray*}
We have: 
\begin{lemm} \label{lemaI}
The function $\zeta_{l,\pm}$ has the properties:
\begin{eqnarray*}
&&\zeta_{l,\pm} \in H^1(\Omega)\cap L^{\infty}(\Omega), \\  
&&\nabla \zeta_{l,\pm} = \nabla e^{T_l(v_l)} \psi_s^{\pm}.
\end{eqnarray*} 
\end{lemm}
\begin{proof}
A simple majoration gives 
$$|\zeta_{l,\pm}(x)| \leq  \int_0^{l}| h_{\pm}(y)|dy, \quad \text{a.e. in } \Omega,$$
hence $\zeta_{l,\pm} \in L^{\infty}(\Omega)$. \\ 
We next remark that $g_{\pm}$ is a Lipschitz function and its classical derivative is given by: 
$$g_{\pm}'(t)= h_{\pm}(T_l(t)) {\mathbf 1}_{\{|t| \leq   l\}} \quad \forall t \neq  l,-l.$$
Hence by using Theorem \ref{stam-lip} we obtain $\zeta_{l,\pm} \in H^1(\Omega)$ and 
\begin{align*}
\nabla \zeta_{l,\pm} &= h_{\pm}(T_l(v_l)) {\mathbf 1}_{\{|v_l| \leq   l\}} \nabla v_l =   h_{\pm}(v_l) {\mathbf 1}_{\{|v_l| \leq   l\}} \nabla v_l  \\ 
&= e^{v_l} \nabla v_l {\mathbf 1}_{\{|v_l| \leq   l\}} \psi_s^{\pm}(v_l)  =  \nabla e^{T_l(v_l)} \psi_s^{\pm}.
\end{align*} 
\end{proof}    

\end{document}